\documentclass[letterpaper,12pt]{amsart}
\usepackage{latexsym,array,delarray,amsmath,amsthm,amssymb,epsfig,setspace,tikz,ulem,enumitem}
\usepackage{cite}
\usepackage{float}
\usepackage{nicematrix}	
\usepackage{import}

\usepackage{color}

\theoremstyle{plain}
\newtheorem{thm}{Theorem}[section]

\newtheorem{lemma}[thm]{Lemma}
\newtheorem{prop}[thm]{Proposition}

\newtheorem{cor}[thm]{Corollary}

\newtheorem*{thm*}{Theorem}
\newtheorem*{lemma*}{Lemma}
\newtheorem*{prop*}{Proposition}
\newtheorem*{cor*}{Corollary}
\newtheorem*{conj*}{Conjecture}

\theoremstyle{definition}
\newtheorem{defn}[thm]{Definition}

\newtheorem{ex}[thm]{Example}

\newtheorem{rmk}[thm]{Remark}

\newcommand{\rr}{\mathbb{R}}

\newcommand{\cm}{\mathcal{M}}

\newcommand{\ind}{\mbox{$\perp \kern-5.5pt \perp$}}

\title{Graph-based sufficient conditions for indistinguishability of linear compartmental models}

\author[Bortner]{Cashous Bortner}
\address{California State University, Stanislaus}
\author[Meshkat]{Nicolette Meshkat}
\address{Santa Clara University}

\date{\today}

\begin{document}
\maketitle

\begin{abstract}
An important problem in biological modeling is choosing the right model.  Given experimental data, one is supposed to find the best mathematical representation to describe the real-world phenomena.  However, there may not be a unique model representing that real-world phenomena.  Two distinct models could yield the same exact dynamics.  In this case, these models are called \textit{indistinguishable}.  In this work, we consider the indistinguishability problem for linear compartmental models, which are used in many areas, such as pharmacokinetics, physiology, cell biology, toxicology, and ecology.  We exhibit sufficient conditions for indistinguishability for models with a certain graph structure: paths from input to output with ``detours''.  The benefit of applying our results is that indistinguishability can be proven using only the graph structure of the models, without the use of any symbolic computation.  This can be very helpful for medium-to-large sized linear compartmental models.  These are the first sufficient conditions for indistinguishability of linear compartmental models based on graph structure alone, as previously only necessary conditions for indistinguishability of linear compartmental models existed based on graph structure alone.  We prove our results by showing that the indistinguishable models are the same up to a renaming of parameters, which we call \textit{permutation indistinguishability}. 
\end{abstract}

\normalem

\section{Introduction} \label{sec:intro}

Two state space models are \emph{indistinguishable} if for any
choice of parameters in the first model, there is a choice of parameters in the second model that will yield the same dynamics in both models, and vice versa.  The indistinguishability problem can be related to the problem of structural identifiability, which concerns determining if all of the unknown parameters in a model can be recovered from input/output data, although it is not necessary to test identifiability in order to test indistinguishability \cite{godfrey-chapman-vajda-1994, walter1984, zhang1991}.

Previous work has focused on testing indistinguishability of given models and also on finding the set of all indistinguishable models from a given model.  Vajda \cite{vajda1982} presents the problem in terms of structural equivalence and gives a method for generating a whole set of indistinguishable models.  Godfrey and Chapman \cite{godfrey-chapman} give necessary conditions for indistinguishability based on the graph structure of the model, which they call geometric rules.  Zhang et al \cite{zhang1991} present a method applying the Godfrey and Chapman geometric rules to generate all identifiable indistinguishable models given a starting model.  Davidson et al \cite{davidson2017} present the program DISTING to compute all indistinguishable models from an identifiable subgraph of a given linear compartmental model.  There has also been some work on indistinguishability of nonlinear compartmental models \cite{godfrey-chapman-vajda-1994, chapman1994, chapman1996}.  Yates et al \cite{yates} review the problems of identifiability and indistinguishability for compartmental models.

While there are necessary conditions for indistinguishability via the graph structure (see the Godfrey-Chapman geometric rules in Section \ref{subsection:gfgeorules}) that can be applied without the use of symbolic computation, to the best our knowledge, there are no sufficient conditions for indistinguishability based on the graph structure. Our goal in this work is to find sufficient conditions for indistinguishability based on graph structure for a certain class of models.  The starting point is a simple \textit{path model} from input to output and we examine the effect of creating and moving a ``detour'' to the path.  We prove our results by showing that the models we consider are indistinguishable via a permutation, or renaming, of parameters.  This renaming of parameters allows us to easily prove indistinguishability. 

{  Path models from an input to an output with one or more leaks arise in physiological models involving metabolism, biliary, or excretory pathways \cite{distefano-book} and models of neuronal dendritic trees \cite{bressloff1993compartmental}.  Some examples of path models with a ``detour" are given in \cite{distefano-book} (see Examples 4.1 and 4.2) where there is a simple path from input to output and there is an exchange pathway that occurs off of the main path (representing the ``detour").  Path models also arise when modeling the delayed response to input and are called
time-delay models and can arise from oral dosing losses and delays in the gastrointestinal tract \cite{distefano-book}.}

{  We were motivated by the following question: given a path from compartment $1$ to compartment $n$, with an input in compartment $1$,  an output in compartment $n$, and a leak from compartment $1$, can we move the leak and get an indistinguishable model? Alternatively, can we change the leak to an edge between two compartments and get an indistinguishable model?   }

Our main results are the following.  Given a path from compartment $1$ to compartment $n$, with input in compartment $1$ and output in compartment $n$, the class of models with a leak from compartment $i$, where $i=1,\ldots,n-1$, are all indistinguishable (see Theorem \ref{thm:leak}).  Replacing the leak with an edge from compartment $n$ to compartment $n-1$ also gives an indistinguishable model (see Theorem \ref{thm:terminalcycle}).  We also generalize Theorem \ref{thm:leak} to apply to moving any ``detour'' down the path (see Theorem \ref{thm:detour}).  Finally, we show how to generalize even further with multiple source-sink paths (see Corollary \ref{cor:sink} and Corollary \ref{cor:source}).

To the best of our knowledge, our results provide the first sufficient conditions for indistinguishability of linear compartmental models based on graph structure alone.  Thus, with our results, indistinguishability can be proven without the use of symbolic computation for a certain class of path models.  

The outline of our paper is the following.  We first give some background on linear compartmental models, input-output equations, and identifiability in Section \ref{section:background}.  We then define indistinguishability and provide many motivating examples in Section \ref{section:indist}. We then provide sufficient conditions for indistinguishability for paths with leaks (Section \ref{section:leak}), paths with detours (Section \ref{section:detour}), and multiple sinks or multiple sources (Section \ref{section:sink} and Section \ref{section:source}).  Finally, we discuss our results and future work in Section \ref{section:discussion}.

\section{Background} \label{section:background} 

\subsection{Linear Compartmental Models}

Let $G$ be a directed graph with vertex set $V$ and set of 
directed edges $E$.  Each vertex $i \in V$ corresponds to a
compartment in our model and an edge $j \rightarrow i$ denotes 
a direct flow of material from compartment $j$ to
compartment $i$.  Also introduce three subsets of the vertices
$In, Out, Leak \subseteq V$ corresponding to the
set of input compartments, output compartments, and leak compartments,
respectively.  To each edge $j \rightarrow i$ we associate
an independent parameter $a_{ij}$, the rate of flow
from compartment $j$ to compartment $i$.  
To each leak node $i \in Leak$, we associate an independent
parameter $a_{0i}$, the rate of flow from compartment $i$ leaving the system.

We associate a matrix $A(G)$, called the \textit{compartmental matrix} {  \cite{linear-i-o}} to the graph and the set $Leak$
 in the following way:
\[
  A(G)_{ij} = \left\{ 
  \begin{array}{l l l}
    -a_{0i}-\sum_{k: i \rightarrow k \in E}{a_{ki}} & \quad \text{if $i=j$ and } i \in Leak\\
        -\sum_{k: i \rightarrow k \in E}{a_{ki}} & \quad \text{if $i=j$ and } i \notin Leak\\
    a_{ij} & \quad \text{if $j\rightarrow{i}$ is an edge of $G$}\\
    0 & \quad \text{otherwise}\\
  \end{array} \right.
\]
For brevity, we will often
 use $A$ to denote $A(G)$. {  Note that when there are no leaks, the compartmental matrix $A(G)$ is the same as the negative Laplacian matrix associated to the graph.  If there are leak terms, then the compartmental matrix $A(G)$ can be described as the negative Laplacian matrix minus a diagonal matrix of leak parameters.  This fact can be useful in proofs, see for example Theorem 2.6 of \cite{MeshkatSullivantEisenberg}.} 

Then we construct a system of linear ODEs with inputs and outputs associated to the quadruple
$(G, In, Out, Leak)$ as follows:
\begin{equation} \label{eq:main}
\dot{x}(t)=Ax(t)+u(t)  \quad \quad y_i(t)=x_i(t)  \mbox{ for } i \in Out
\end{equation}
 where $u_{i}(t) \equiv 0$ for $i \notin In$.
 The coordinate functions $x_{i}(t)$ are the state variables, the 
 functions $y_{i}(t)$ are the output variables, and the nonzero functions $u_{i}(t)$ are
 the inputs.  The resulting model is called a   \textit{linear compartmental model}. 

We will indicate output compartments by this symbol: \begin{tikzpicture}[scale=0.7]
 	\draw (4.66,-.49) circle (0.05);	
	 \draw[-] (5, -.15) -- (4.7, -.45);	
\end{tikzpicture} .  
Input compartments are labeled by ``in'', and leaks are indicated by edges which go to no vertex.

\begin{figure} { 
    \centering
\begin{tikzpicture}[scale=.9]
 	\draw (0,0) circle (0.3);
 	\draw (2,0) circle (0.3);
    	\node[] at (0, 0) {1};
    	\node[] at (2, 0) {2};
	 \draw[->] (0.35, .1) -- (1.65, .1);
	 \draw[<-] (0.35, -.1) -- (1.65, -.1);
	 
   	 \node[] at (1, 0.35) {$a_{21}$};
   	 \node[] at (1,-.35) {$a_{12}$};
	\draw (2.69,.69) circle (0.07);	
	 \draw[-] (2.65, .65 ) -- (2.22, .22);
	 \draw[->] (-.65, .65) -- (-.25, .25);	
   	 \node[] at (-.8,.8) {in};
	 \draw[->] (0,-.3) -- (0, -.9);	
   	 \node[] at (0.35, -.7) {$a_{01}$};
\draw (-1.4,-2) rectangle (3, 2);
    	\node[] at (1, -1.5) {$\cm$};
    	

\end{tikzpicture}
    \caption{Model described in Example \ref{ex:continuing}.}
    \label{fig:firstexample}}
\end{figure}
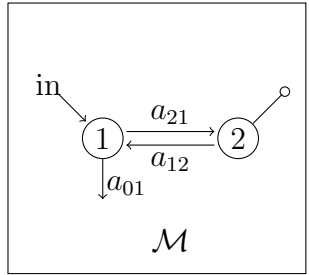

\begin{ex} \label{ex:continuing} The model $\cm=(G,\{1\},\{2\},\{1\})$ with $G$ given in { Figure \ref{fig:firstexample}} 
is a linear compartmental model with equations given by:
{\small
\begin{align} \label{eq:A-ex}
\begin{pmatrix}
\dot{x}_1 \\
\dot{x}_2 
\end{pmatrix} 
&~=~
\begin{pmatrix}
- a_{01} - a_{21} & a_{12} \\
a_{21} & -a_{12}
\end{pmatrix}
\begin{pmatrix}
x_1 \\
x_2
\end{pmatrix} +
\begin{pmatrix}
u_1 \\
0
\end{pmatrix}~,
\end{align}
}

with 
output equation $y_2=x_2$. 
\end{ex}

\begin{defn} A \textit{path} from vertex $i_{0}$ to vertex $i_{k}$ in a directed graph $G$ is a sequence of 
vertices $i_{0},i_{1}, i_{2}, \ldots, i_{k}$
such that $i_{j} \to i_{j+1}$ is an edge for all $j = 0, \ldots, k-1$.
\end{defn}

Now we give some definitions from \cite{linear-i-o} regarding an important subgraph to this work:

\begin{defn} \label{defn:outputreachable} For a linear compartmental model $\cm=(G,In,Out,Leak)$, let $i \in Out$.  The \textit{output-reachable subgraph to $i$} (or \textit{to $y_i$}) is the induced subgraph of $G$ containing all vertices $j$ for which there is a directed path in $G$ from $j$ to $i$. A linear compartmental model is \textit{output connectable} if every compartment has a directed path leading from it to an output compartment. Likewise, a linear compartmental model is \textit{input connectable} if every compartment has a directed path leading to it from an input compartment.
\end{defn}


\subsection{Input-output equations}

We will be using the so-called \textit{differential algebra approach} to find \textit{input-output equations} to analyze both structural identifiability \cite{Ljung,Ollivier} and indistinguishability.  In the differential algebra approach, we view the model equations as differential polynomials in a differential polynomial ring $R(p)[u,y,x]$, i.e., the ring of polynomials in state variable vector $x$, output vector $y$, input vector $u$, and their derivatives, with coefficients in $R(p)$ for parameter vector $p$. Since the unmeasured state variables $x_i$ cannot be determined, the goal in this approach is to use differential elimination to eliminate all unknown state variables and their derivatives.  The resulting equations are only in terms of input variables, output
variables, their derivatives, and parameters, so these equations have the following form:
\begin{align} \label{eq:general-i-o}
\sum_i{c_i(p)\Psi_i(u,y)} =0~
\end{align}
{ where $c_i(p)$ and $\Psi_i(u,v)$ are defined below.}  An equation of the form~\eqref{eq:general-i-o} is called an \textit{input-output equation} for
$\mathcal M$. 

For nonlinear models, one standard ``reduced'' generating set for these input-output equations is formed by those equations in a \textit{characteristic set} (defined precisely in~\cite{glad}{ , but briefly, it is a triangular set in reduced form, according to an ordering of variables})
that do not involve the $x_i$'s or their derivatives.
In a characteristic set, which can be computed using the software {\tt DAISY}~\cite{daisy},
each $\Psi_i(u,y)$ in each input-output equation~\eqref{eq:general-i-o} 
is a differential monomial, i.e., a monomial purely
in terms of input variables, output variables, and their derivatives. The terms $c_i(p)$ are called the coefficients of the input-output equations.  These coefficients can be fixed uniquely by normalizing the input-output equations to make them monic \cite{daisy}.

However, for linear models, it has been shown that these input-output equations can be found much more easily using the Transfer Function approach \cite{Bellman} or even a trick with Cramer's Rule \cite{MeshkatSullivant}.  We will be taking the latter approach to get an explicit formula for the input-output equations.  

We now state Theorem 3.8 from \cite{linear-i-o} with input $i$ and output $j$, which gives the input-output equation in $y_j$ in terms of the output-reachable subgraph to $y_j$.

\begin{thm} \label{thm:ioscc} 
Let $\mathcal{M}=(G, In, Out, Leak)$
be a linear compartmental model with at least one input.
Let $j \in Out$, 
and assume that there exists a directed path from some input compartment 
to compartment-$j$. 
Let $ H$ denote the output-reachable subgraph to $y_j$, let $A_H$ denote the compartmental matrix for the restriction $\mathcal{M}_H$, and let $\partial I$ be the the product of the differential operator $d/dt$ and the $|V_G| \times |V_G|$ identity matrix.  
Then the following is an input-output equation for $\mathcal M$ involving $y_j$:
 \begin{align}  \label{eq:io-det}
 	\det (\partial I -{A}_H) y_j ~=~   \sum_{i \in In \cap V_H} (-1)^{i+j} \det \left( \partial I-{A}_H \right)^{i,j} u_i ~,
	 \end{align}
where $ \left( \partial I-{A}_H \right)^{i,j}$ denotes the matrix obtained from
 $\left( \partial I-{A}_H \right)$ by removing the row corresponding to compartment-$i$ and the column corresponding to compartment-$j$. 
\end{thm}

{  \begin{rmk} In this work, we will only be considering output connectable models, thus we will be using Theorem \ref{thm:ioscc} heavily to obtain an explicit formula for the input-output equations in terms of the output-reachable subgraph.
\end{rmk}}

\begin{ex} [Continuation of Example \ref{ex:continuing}]\label{ex:ioeq}
 The model $\cm=(G,\{1\},\{2\},\{1\})$ with $G$ given in Figure \ref{fig:nonpermute} 
 has the following input-output equation: 

 {  
{\small
\begin{align*} 
\det
\begin{pmatrix}
\partial + a_{01} + a_{21} & - a_{12} \\
-a_{21} & \partial + a_{12}
\end{pmatrix}
y_2
=
(-1)^3
\det
\begin{pmatrix}
-a_{21}
\end{pmatrix}~
u_1,
\end{align*}
}
which reduces to:
}
\begin{align*} 
	\ddot{y}_2+(a_{01}+a_{21}+a_{12})\dot{y}_2+(a_{01}a_{12})y_2 = a_{21}u_1
\end{align*}

\end{ex}

\begin{rmk} We note that the input-output equation has $y_j$ and derivatives of $y_j$ on the left-hand side and $u_i$ and  derivatives of $u_i$ on the right-hand side.  For brevity we collectively refer to them as derivatives of $y_j$ and $u_i$.  We also note there are polynomial functions of parameters as coefficients of these derivatives, appearing on both sides.  Therefore, in the rest of the paper, we will refer to the \textit{right-hand side coefficients} and the \textit{left-hand side coefficients} as the polynomial functions of parameters that appear in front of the derivatives of $u_i$ and $y_j$, respectively. 
\end{rmk}

\subsection{Identifiability}

We will see shortly that sometimes the question of identifiability can help answer questions about indistinguishability.  We give some definitions of identifiability from \cite{MeshkatSullivantEisenberg}. {  Recent work includes \cite{hong2020global, ovchinnikov2021computing, renardy2022structural}.}

\begin{defn}\label{defn:identify}
Let $(G, In, Out, Leak)$ be a linear compartment model and
let $c$ denote the vector of all nonzero and nonmonic coefficient functions of
all the linear input-output equations derived in Theorem
\ref{thm:ioscc} for each $i \in Out$.  The function $c$ defines a map
$c:  \rr^{|E| + |Leak|}  \rightarrow \rr^{k}$,
where $k$ is the total number of coefficients which we call the \textit{coefficient map}.  The linear compartment model  $(G, In, Out, Leak)$ is: 
\begin{itemize}
	\item{\textit{globally identifiable} if $c$ is a one-to-one function, and is \textit{generically globally identifiable} if global identifiability holds everywhere in $\rr^{|E| + |Leak|}$, except possibly on a set of measure zero.}
	\item{\textit{locally identifiable} if around any neighborhood of a point in $\rr^{|E| + |Leak|}$, $c$ is a one-to-one function, and is \textit{generically locally identifiable} if local identifiability holds everywhere in $\rr^{|E| + |Leak|}$, except possibly on a set of measure zero.}
\item{\textit{unidentifiable} if $c$ is infinite-to-one.}
\end{itemize}
\end{defn}

\begin{ex} [Continuation of Example \ref{ex:continuing}]\label{ex:ident}
 The model $\cm=(G,\{1\},\{2\},\{1\})$ with $G$ given in Figure \ref{fig:nonpermute} 
 has the following coefficient map: 
\begin{align*} 
	c(a_{12}, a_{21}, a_{01})= (a_{01}+a_{21}+a_{12}, a_{01}a_{12}, a_{21})
\end{align*}

We can test injectivity of this map by solving the system of polynomial equations:
\begin{align*} 
	a_{01}+a_{21}+a_{12} = a^*_{01}+a^*_{21}+a^*_{12} \\ a_{01}a_{12} = a^*_{01}a^*_{12} \\
 a_{21}=a^*_{21}
\end{align*}

We get two solutions:
\begin{align*} 
	(a_{12}, a_{21}, a_{01})=(a^*_{12}, a^*_{21}, a^*_{01}) \text{ or} \\
 (a_{12}, a_{21}, a_{01})=(a^*_{01}, a^*_{21}, a^*_{12})
\end{align*}

Thus the model is locally identifiable.
\end{ex}

\section{Indistinguishability} \label{section:indist}

Recall two state space models are \emph{indistinguishable} if for any
choice of parameters in the first model, there is a choice of parameters
in the second model that will yield the same dynamics in both models, and
vice versa.  There have been several definitions and approaches to solve this problem in the literature \cite{godfrey-chapman,raksanyi,walter1984,zhang1991}.

To start with, to be indistinguishable, two models must have the same
input and output variables.  Since indistinguishable models give the same
dynamics, the structures of their input-output equations should be the same.
In the case that there is one output variable in both models, there is a single
input-output equation.  To say the input-output equations  have the same structure means that
exactly the same differential monomials (i.e. $y_i$, $\dot{y}_i$, ...,$u_i$, $\dot{u}_i$, ..., etc) appear in both input-output equations.

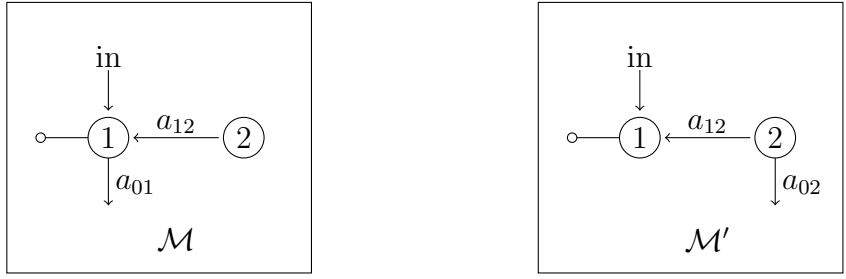
\begin{figure} 
\begin{center}

\begin{tikzpicture}[scale=.9]
 	\draw (0,0) circle (0.3);
 	\draw (2,0) circle (0.3);
    	\node[] at (0, 0) {1};
    	\node[] at (2, 0) {2};
	 \draw[<-] (0.37, 0) -- (1.63, 0);
   	 \node[] at (1, 0.2) {$a_{12}$};
	\draw (-1,0) circle (0.07);	
	 \draw[-] ( -.3,0 ) -- ( -.93,0);
	 \draw[->] (0, 1) -- (0, .4);	
   	 \node[] at (0,1.2) {in};
	 \draw[->] (0,-.3) -- (0,-1);	
   	 \node[] at (.4, -.7) {$a_{01}$};
\draw (-1.5,-2) rectangle (3, 2);
    	\node[] at (1, -1.5) {$\cm$};
    	

\end{tikzpicture}\hspace*{3cm}\begin{tikzpicture}[scale=.9]
 	\draw (0,0) circle (0.3);
 	\draw (2,0) circle (0.3);
    	\node[] at (0, 0) {1};
    	\node[] at (2, 0) {2};
	 \draw[<-] (0.37, 0) -- (1.63, 0);
   	 \node[] at (1, 0.2) {$a_{12}$};
	\draw (-1,0) circle (0.07);	
	 \draw[-] ( -.3,0 ) -- ( -.93,0);
	 \draw[->] (0, 1) -- (0, .4);	
   	 \node[] at (0,1.2) {in};
	 \draw[->] (2,-.3) -- (2,-1);	
   	 \node[] at (2.4, -.7) {$a_{02}$};

\draw (-1.5,-2) rectangle (3, 2);
    	\node[] at (1, -1.5) {$\cm'$};
    	

\end{tikzpicture}
\end{center}

\caption{Two models described in Example \ref{ex:small2}}
\label{fig:small2}
\end{figure}
\begin{ex}\label{ex:small2} [\textbf{Distinguishable via input-output equation structure}]
Consider the models given in Figure \ref{fig:small2}. 
The model $\cm$ has the form:
   \begin{align*}
       \dot{x}_1 = -a_{01} x_1 + a_{12} x_2 + u_1\\
       \dot{x}_2 = -a_{12}x_2 \\
       y_1 = x_1
   \end{align*} 
   while the model $\cm'$ has the form:
   \begin{align*}
       \dot{x}_1 = + a_{12} x_2 + u_1\\
       \dot{x}_2 = -(a_{02} + a_{12})x_2 \\
       y_1 = x_1
   \end{align*}

   The input-output equations for $\cm$ and $\cm'$ are, respectively:
   \begin{align*}
       \ddot{y}_1 + (a_{01} + a_{12}) \dot{y}_1 + a_{01}a_{12} y_1 = \dot{u}_1 + a_{12} u_1  \\
       \ddot{y}_1 + (a_{02} + a_{12}) \dot{y}_1 = \dot{u}_1 + (a_{02} + a_{12}) u_1 
   \end{align*}

   Since these two models do not have the same structure of the input-output equations, e.g. $\cm$ has a term in $y_1$ while $\cm'$ does not, then these two models are \textit{distinguishable}. 
\end{ex}

Supposing that the two models have the same structures as described above,
we can let $c(p)$ and $c'(p')$ denote the corresponding coefficient maps of
the two models, respectively.  Here $c : \Theta \rightarrow \rr^m$ and 
$c':  \Theta' \rightarrow \rr^m$, and the components are ordered so that the
components correspond to each other as coming from the same differential
monomial.   Note that the dimensions of the parameter spaces $\Theta$ and $\Theta'$
might be different. We further assume that both coefficient maps are monic on the same coefficient. Indistinguishability is characterized in terms of
the coefficient maps $c$ and $c'$.

\begin{defn}
Suppose that Model 1 and Model 2 have the same input-output equation structure.
Let $c:  \Theta \rightarrow \rr^{m}$ and $c':  \Theta^{'} \rightarrow \rr^{m}$ 
be the coefficient maps for Model 1 and Model 2, respectively.  We say that:
\begin{itemize}
\item  Model 1 and Model 2 are \emph{indistinguishable} if
for all  $p' \in \Theta'$, there exists at least one  
$p\in \Theta$ such that $c(p)=c'(p')$, and vice versa;
\item  Model 1 and Model 2 are \emph{distinguishable}
if they are not indistinguishable.
\end{itemize}
\end{defn}

\begin{rmk}
The definition of indistinguishability is equivalent to saying that
$c(\Theta) =  c'(\Theta')$, i.e. the images of the coefficient maps are the same.  
\end{rmk}

A simple observation on distinguishability is that indistinguishable models
must have the same vanishing ideal on the image of the parametrization \cite{meshkat-rosen-sullivant}.  This means that the coefficients must satisfy the same algebraic dependency relationships.
This is usually easy to check in small to medium-sized examples.

\begin{figure} 
\begin{center}

\begin{tikzpicture}[scale=.8]
 	\draw (0,0) circle (0.3);
 	\draw (2,0) circle (0.3);
 	\draw (4,0) circle (0.3);
    	\node[] at (0, 0) {1};
    	\node[] at (2, 0) {2};
    	\node[] at (4, 0) {$3$};
	 \draw[->] (0.35, .1) -- (1.65, .1);
	 \draw[->] (2.35,.1) -- (3.65,.1);
   	 \node[] at (1, 0.3) {$a_{21}$};
	\node[] at (3,0.3) {$a_{32}$};
	\draw (4.69,.69) circle (0.07);	
	 \draw[-] (4.65, .65 ) -- (4.22, .22);
	 \draw[->] (-.65, .65) -- (-.25, .25);	
   	 \node[] at (-.87,.87) {in};
	 \draw[->] (2, .85) -- (2, .35);	
   	 \node[] at (2,1.1) {in};

	 \draw[->] (0,-.3) -- (0, -.9);	
   	 \node[] at (0.4, -.7) {$a_{01}$};
\draw (-1.4,-2) rectangle (5, 2);
    	\node[] at (2, -1.5) {$\cm$};
    	

\end{tikzpicture}
\begin{tikzpicture}[scale=.8]
 	\draw (0,0) circle (0.3);
 	\draw (2,0) circle (0.3);
 	\draw (4,0) circle (0.3);
    	\node[] at (0, 0) {1};
    	\node[] at (2, 0) {2};
    	\node[] at (4, 0) {$3$};
	 \draw[->] (0.35, .1) -- (1.65, .1);
	 \draw[->] (2.35,.1) -- (3.65,.1);
   	 \node[] at (1, 0.3) {$a_{21}$};
	\node[] at (3,0.3) {$a_{32}$};
	\draw (4.69,.69) circle (0.07);	
	 \draw[-] (4.65, .65 ) -- (4.22, .22);
	 \draw[->] (-.65, .65) -- (-.25, .25);	
   	 \node[] at (-.87,.87) {in};
	 \draw[->] (2, .85) -- (2, .35);	
   	 \node[] at (2,1.1) {in};

	 \draw[->] (2,-.3) -- (2, -.9);	
   	 \node[] at (2.4, -.7) {$a_{02}$};
\draw (-1.4,-2) rectangle (5, 2);
    	\node[] at (2, -1.5) {$\cm'$};
    	

\end{tikzpicture}

\caption{Indistinugishable models described in Example \ref{ex:twoinputdist}.}
\label{fig:twoinputdist}
\end{center}
\end{figure}
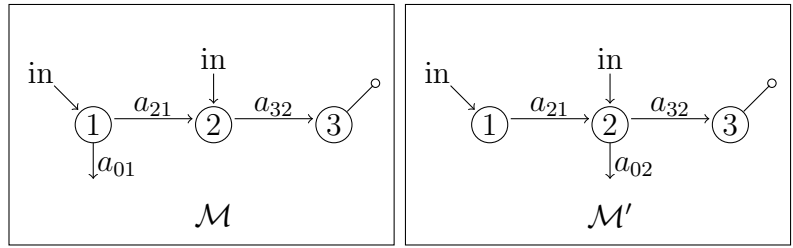

\begin{ex}\label{ex:twoinputdist} [\textbf{Distinguishable via coefficient dependency relationships}]
    Consider the following models in Figure \ref{fig:twoinputdist}. The model $\cm$ has the form:
    \begin{align*}
        \dot{x}_1 = -(a_{01}+a_{21})x_1 + u_1 \\
        \dot{x}_2 = a_{21}x_1 -a_{32}x_2 + u_2 \\
        \dot{x}_3 = a_{32}x_2 \\
        y_3 = x_3
    \end{align*}
while the model $\cm'$ has the form:
\begin{align*}
        \dot{x}_1 = -a_{21}x_1 + u_1 \\
        \dot{x}_2 = a_{21}x_1 -(a_{02} + a_{32})x_2 + u_2 \\
        \dot{x}_3 = a_{32}x_2 \\
        y_3 = x_3
    \end{align*}
    
The input-output equations for $\cm$ and $\cm'$ are, respectively:
\begin{align*}
    \dddot{y}_3 + \underbrace{(a_{32} + a_{01} + a_{21})}_{ c_1}\ddot{y}_3 + \underbrace{(a_{01}a_{32} + a_{21}a_{32})}_{ c_2}\dot{y}_3 = \underbrace{a_{21}a_{32}}_{ c_3}u_1 + \underbrace{a_{32}}_{ c_4}\dot{u}_2 + \underbrace{(a_{01}a_{32} + a_{21}a_{32})}_{ c_5} u_2 \\
    \dddot{y}_3 + \underbrace{(a_{32} + a_{02} + a_{21})}_{ c_1'}\ddot{y}_3 + \underbrace{(a_{02}a_{21} + a_{21}a_{32})}_{ c_2'}\dot{y}_3 = \underbrace{a_{21}a_{32}}_{ c_3'}u_1 + \underbrace{a_{32}}_{ c_4'}\dot{u}_2 + \underbrace{a_{21}a_{32}}_{ c_5'} u_2
\end{align*}

If we label the coefficients $c_1,...,c_5$ for $\cm$ and $c'_1,...,c'_5$ for $\cm'$, then the algebraic dependency relationships among the coefficients for $\cm$ and $\cm'$ are the following, respectively:
\begin{align*}
    c_2 - c_5 = 0, c_1c_4 - c_4^2 -c_5 = 0 \\
    c'_3 - c'_5 = 0, c'_2{c'_4}^2 -c'_1c'_4c'_5 + {c'_5}^2 =0
\end{align*}

Since the algebraic dependency relationships among the coefficients are different, the models are distinguishable.
    
\end{ex}

Once the same vanishing ideal has been established, 
an approach for checking indistinguishability is
to construct the equation system $c(p) = c'(p')$ and attempt to
``solve'' for one set of parameters in terms of the other, and vice versa, using
Gr\"obner basis calculations \cite{meshkat-rosen-sullivant}.  Once this has been done, one must check the
resulting solutions to determine if they satisfy the necessary inequality
constraints of the parameter spaces $\Theta$ and $\Theta'$.  We note that identifiable models with coefficient maps satisfying the same algebraic dependence relationships can always be solved for one set of parameters in terms of the other, and vice versa, {  since the coefficient maps are injective maps to the same image, and thus invertible over this restriction, }but the parameter constraints (e.g.~positive parameters) must still be checked for indistinguishability to hold \cite{meshkat-rosen-sullivant}.

While the method above gives an approach to test indistinguishability, we note that this approach involves symbolic computation, which can be computationally expensive for medium to large-sized examples.  Thus, in this work, we seek out a different approach to test indistinguishability.  The motivation comes from the next example, where we consider indistinguishability of linear compartmental models via a permutation of parameters.  Thus, rather than testing that the images of the coefficient maps are the same using Gr\"obner basis calculations, we can simply show that the coefficient maps are identical up to a renaming of parameters:

\begin{figure} 
\begin{center}

\begin{tikzpicture}[scale=.8]
 	\draw (0,0) circle (0.3);
 	\draw (2,0) circle (0.3);
 	\draw (4,0) circle (0.3);
    	\node[] at (0, 0) {1};
    	\node[] at (2, 0) {2};
    	\node[] at (4, 0) {$3$};
	 \draw[->] (0.35, .1) -- (1.65, .1);
	 \draw[->] (2.35,.1) -- (3.65,.1);
   	 \node[] at (1, 0.3) {$a_{21}$};
	\node[] at (3,0.3) {$a_{32}$};
	\draw (4.69,.69) circle (0.07);	
	 \draw[-] (4.65, .65 ) -- (4.22, .22);
	 \draw[->] (-.65, .65) -- (-.25, .25);	
   	 \node[] at (-.8,.8) {in};
	 \draw[->] (0,-.3) -- (0, -.9);	
   	 \node[] at (0.35, -.7) {$a_{01}$};
\draw (-1.4,-2) rectangle (5, 2);
    	\node[] at (2, -1.5) {$\cm$};
    	

\end{tikzpicture}
\begin{tikzpicture}[scale=.8]
 	\draw (0,0) circle (0.3);
 	\draw (2,0) circle (0.3);
 	\draw (4,0) circle (0.3);
    	\node[] at (0, 0) {1};
    	\node[] at (2, 0) {2};
    	\node[] at (4, 0) {$3$};
	 \draw[->] (0.35, .1) -- (1.65, .1);
	 \draw[->] (2.35,.1) -- (3.65,.1);
   	 \node[] at (1, 0.3) {$a_{21}$};
	\node[] at (3,0.3) {$a_{32}$};
	\draw (4.69,.69) circle (0.07);	
	 \draw[-] (4.65, .65 ) -- (4.22, .22);
	 \draw[->] (-.65, .65) -- (-.25, .25);	
   	 \node[] at (-.8,.8) {in};
	 \draw[->] (2,-.3) -- (2, -.9);	
   	 \node[] at (2.35, -.7) {$a_{02}$};
\draw (-1.4,-2) rectangle (5, 2);
    	\node[] at (2, -1.5) {$\cm'$};
    	

\end{tikzpicture}
\begin{tikzpicture}[scale=.8]
 	\draw (0,0) circle (0.3);
 	\draw (2,0) circle (0.3);
 	\draw (4,0) circle (0.3);
    	\node[] at (0, 0) {1};
    	\node[] at (2, 0) {2};
    	\node[] at (4, 0) {$3$};
	 \draw[->] (0.35, .1) -- (1.65, .1);
	 \draw[->] (2.35,.1) -- (3.65,.1);
	 \draw[<-] (2.35,-.1) -- (3.65,-.1);
   	 \node[] at (1, 0.3) {$a_{21}$};
	\node[] at (3,0.3) {$a_{32}$};
	\node[] at (3,-.3) {$a_{23}$};
	\draw (4.69,.69) circle (0.07);	
	 \draw[-] (4.65, .65 ) -- (4.22, .22);
	 \draw[->] (-.65, .65) -- (-.25, .25);	
   	 \node[] at (-.8,.8) {in};
\draw (-1.4,-2) rectangle (5, 2);
    	\node[] at (2, -1.5) {$\cm''$};
    	

\end{tikzpicture}

\end{center}

\caption{Three permutation indistinguishable models described in Example \ref{ex:permute}}
\label{fig:permute}
\end{figure}
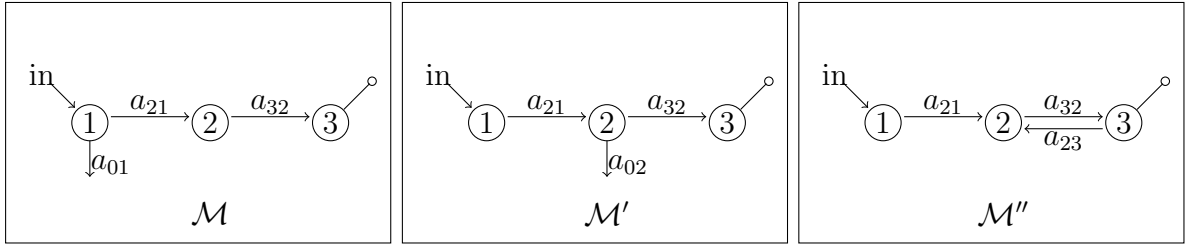
\begin{ex} [\textbf{Indistinguishable via permutation of parameters}] \label{ex:permute}  Consider the models in Figure \ref{fig:permute}.

The model $\cm$ has the form:
    \begin{align*}
        \dot{x}_1 = -(a_{01}+a_{21})x_1 + u_1 \\
        \dot{x}_2 = a_{21}x_1 -a_{32}x_2  \\
        \dot{x}_3 = a_{32}x_2 \\
        y_3 = x_3
    \end{align*}
while the model $\cm'$ has the form:
\begin{align*}
        \dot{x}_1 = -a_{21}x_1 + u_1 \\
        \dot{x}_2 = a_{21}x_1 -(a_{02} + a_{32})x_2  \\
        \dot{x}_3 = a_{32}x_2 \\
        y_3 = x_3
    \end{align*}
and the model $\cm''$ has the form:
\begin{align*}
        \dot{x}_1 = -a_{21}x_1 + u_1 \\
        \dot{x}_2 = a_{21}x_1 -a_{32}x_2 + a_{23}x_3 \\
        \dot{x}_3 = a_{32}x_2 -a_{23}x_3 \\
        y_3 = x_3
    \end{align*}

The input-output equations are of the form:
\begin{align*}
    \cm: &\quad  \dddot{y}_3+(a_{32}+a_{01}+a_{21})\ddot{y}_3+(a_{01}a_{32}+a_{21}a_{32})\dot{y}_3 = (a_{21}a_{32})u_1\\
    \cm': & \quad \dddot{y}_3+(a_{21}+a_{02}+a_{32})\ddot{y}_3+(a_{02}a_{21}+a_{32}a_{21})\dot{y}_3 = (a_{32}a_{21})u_1\\
    \cm'': & \quad \dddot{y}_3+(a_{21}+a_{23}+a_{32})\ddot{y}_3+(a_{23}a_{21}+a_{32}a_{21})\dot{y}_3 = (a_{32}a_{21})u_1
\end{align*}

The three coefficient maps are exactly the same up to a renaming of parameters, i.e., we can rename the parameters of $\cm$ (first column), $\cm'$ (second column), $\cm''$ (third column), respectively: 
\begin{align*}
a_{32} \equiv a_{21} \equiv a_{21}\\
a_{01} \equiv a_{02} \equiv a_{23}\\
a_{21} \equiv a_{32} \equiv a_{32}
\end{align*}

Thus we can see these models are indistinguishable, via this renaming of parameters.
\end{ex} 

From Example \ref{ex:permute}, it appears that one way we can go from model to model is by moving the leak.  However, the situation is a little bit more nuanced as the next example shows:

\begin{figure}
\begin{tikzpicture}[scale=.8]
 	\draw (0,0) circle (0.3);
 	\draw (2,0) circle (0.3);
 	\draw (4,0) circle (0.3);
    	\node[] at (0, 0) {1};
    	\node[] at (2, 0) {2};
    	\node[] at (4, 0) {$3$};
	 \draw[->] (0.35, .1) -- (1.65, .1);
	 \draw[->] (2.35,.1) -- (3.65,.1);
   	 \node[] at (1, 0.3) {$a_{21}$};
	\node[] at (3,0.3) {$a_{32}$};
	\draw (4.69,.69) circle (0.07);	
	 \draw[-] (4.65, .65 ) -- (4.22, .22);
	 \draw[->] (-.65, .65) -- (-.25, .25);	
   	 \node[] at (-.8,.8) {in};
	 \draw[->] (4,-.3) -- (4, -.9);	
   	 \node[] at (4.35, -.7) {$a_{03}$};
\draw (-1.4,-2) rectangle (5, 2);
    	\node[] at (2, -1.5) {$\cm'''$};
    	

\end{tikzpicture}
\caption{A model similar to those in Example \ref{ex:permute}/Figure \ref{fig:permute}, but not indistinguishable as described in Example \ref{ex:movingleaknonindist}.}
\label{fig:movingleaknonindist}
\end{figure}
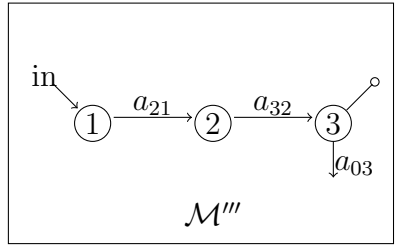

\begin{ex}\label{ex:movingleaknonindist} [\textbf{Moving leak does not always result in indistinguishability}]

Consider now the model $\cm'''$ in Figure \ref{fig:movingleaknonindist}, where we move a leak from compartment 2 in $\cm'$ in Figure \ref{fig:permute} 
to compartment 3 in $\cm'''$.  


The model $\cm'''$ has the form:
\begin{align*}
        \dot{x}_1 = -a_{21}x_1 + u_1 \\
        \dot{x}_2 = a_{21}x_1 -a_{32}x_2  \\
        \dot{x}_3 = a_{32}x_2 -a_{03}x_3 \\
        y_3 = x_3
    \end{align*}

Here the structure of the input-output equation is different from the models $\cm, \cm', \cm''$, as the input-output equation has a constant term in $y_3$:  
\begin{align*}
   \cm''': & \quad \dddot{y}_3+(a_{21}+a_{32}+a_{03})\ddot{y}_3+(a_{03}a_{32}+a_{32}a_{21})\dot{y}_3 + (a_{21}a_{32}a_{03})y_3 = (a_{32}a_{21})u_1\\
\end{align*}
Thus, $\cm'''$ is distinguishable from models $\cm, \cm', \cm''$.  

\end{ex}

While indistinguishability from a permutation of parameters is easy to check, we note that it is possible to have indistinguishable models that do not result from a permutation of parameters.  Consider the following example:

\begin{figure} 
\begin{center}

\begin{tikzpicture}[scale=.9]
 	\draw (0,0) circle (0.3);
 	\draw (2,0) circle (0.3);
    	\node[] at (0, 0) {1};
    	\node[] at (2, 0) {2};
	 \draw[->] (0.35, .1) -- (1.65, .1);
	 \draw[<-] (0.35, -.1) -- (1.65, -.1);
	 
   	 \node[] at (1, 0.35) {$a_{21}$};
   	 \node[] at (1,-.35) {$a_{12}$};
	\draw (2.69,.69) circle (0.07);	
	 \draw[-] (2.65, .65 ) -- (2.22, .22);
	 \draw[->] (-.65, .65) -- (-.25, .25);	
   	 \node[] at (-.8,.8) {in};
	 \draw[->] (0,-.3) -- (0, -.9);	
   	 \node[] at (0.35, -.7) {$a_{01}$};
\draw (-1.4,-2) rectangle (3, 2);
    	\node[] at (1, -1.5) {$\cm$};
    	

\end{tikzpicture}\hspace*{1cm}
\begin{tikzpicture}[scale=.9]
 	\draw (0,0) circle (0.3);
 	\draw (2,0) circle (0.3);
    	\node[] at (0, 0) {1};
    	\node[] at (2, 0) {2};
	 \draw[->] (0.35, .1) -- (1.65, .1);
	 \draw[<-] (0.35, -.1) -- (1.65, -.1);
	 
   	 \node[] at (1, 0.35) {$a_{21}$};
   	 \node[] at (1,-.35) {$a_{12}$};
	\draw (2.69,.69) circle (0.07);	
	 \draw[-] (2.65, .65 ) -- (2.22, .22);
	 \draw[->] (-.65, .65) -- (-.25, .25);	
   	 \node[] at (-.8,.8) {in};
	 \draw[->] (2,-.3) -- (2, -.9);	
   	 \node[] at (2.35, -.7) {$a_{02}$};
\draw (-1.4,-2) rectangle (3, 2);
    	\node[] at (1, -1.5) {$\cm'$};
    	

\end{tikzpicture}\hspace*{1cm}
\begin{tikzpicture}[scale=.9]
 	\draw (0,0) circle (0.3);
 	\draw (2,0) circle (0.3);
    	\node[] at (0, 0) {1};
    	\node[] at (2, 0) {2};
	 \draw[->] (0.35, .1) -- (1.65, .1);
	 
   	 \node[] at (1, 0.35) {$a_{21}$};
	\draw (2.69,.69) circle (0.07);	
	 \draw[-] (2.65, .65 ) -- (2.22, .22);
	 \draw[->] (-.65, .65) -- (-.25, .25);	
   	 \node[] at (-.8,.8) {in};
	 \draw[->] (0,-.3) -- (0, -.9);	
   	 \node[] at (0.35, -.7) {$a_{01}$};
   	 \draw[->] (2,-.3) -- (2, -.9);	
   	 \node[] at (2.35, -.7) {$a_{02}$};
\draw (-1.4,-2) rectangle (3, 2);
    	\node[] at (1, -1.5) {$\cm''$};
    	

\end{tikzpicture}

\end{center}
\caption{Three non-permutation indistinguishable models described in Example \ref{ex:nonpermute}}
\label{fig:nonpermute}
\end{figure}
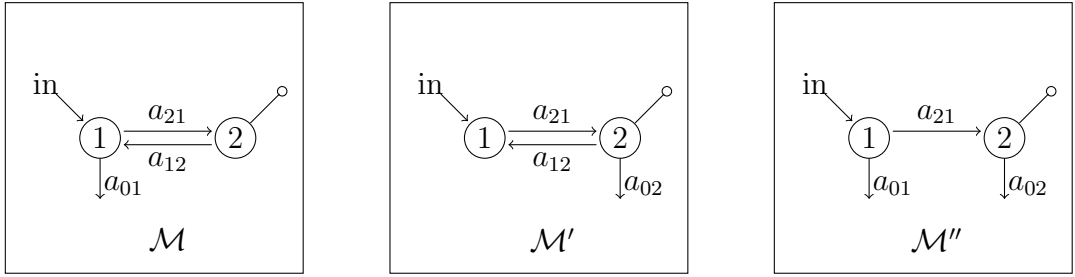
\begin{ex}\label{ex:nonpermute} [\textbf{Indistinguishable, but not from permutation of parameters}]
Consider the models in Figure \ref{fig:nonpermute}.


The model $\cm$ has the form:
    \begin{align*}
        \dot{x}_1 = -(a_{01}+a_{21})x_1 + a_{12}x_2 + u_1 \\
        \dot{x}_2 = a_{21}x_1 { -a_{12}x_2}  \\
        y_2 = x_2
    \end{align*}
while the model $\cm'$ has the form:
\begin{align*}
        \dot{x}_1 = { -}a_{21}x_1 + a_{12}x_2 + u_1 \\
        \dot{x}_2 = a_{21}x_1 - (a_{02} + a_{12})x_2 \\
        y_2 = x_2
    \end{align*}
and the model $\cm''$ has the form:
\begin{align*}
        \dot{x}_1 = -(a_{01}+a_{21})x_1 + u_1 \\
        \dot{x}_2 = a_{21}x_1 -a_{02}x_2  \\
        y_2 = x_2
    \end{align*}

The input-output equations are of the form:
\begin{align*}
    \cm: &\quad  \ddot{y}_2+(a_{01}+a_{21}+a_{12})\dot{y}_2+(a_{01}a_{12})y_2 = a_{21}u_1\\
    \cm': &\quad  \ddot{y}_2+(a_{02}+a_{21}+a_{12})\dot{y}_2+(a_{02}a_{21})y_2 = a_{21}u_1\\
    \cm'': &\quad  \ddot{y}_2+(a_{01}+a_{21}+a_{02})\dot{y}_2+(a_{01}a_{02}+a_{21}a_{02})y_2 = a_{21}u_1\\
\end{align*}
We see that $a_{21}$ must play the same role in all three models.  

However, there is no renaming of parameters that will map one coefficient map to the next.  Nevertheless, these three models are indistinguishable.  We showed in Example \ref{ex:ident} that model $\cm$ is locally identifiable.  Likewise, it can easily be shown that $\cm'$ and $\cm''$ are also (at least) locally identifiable.  Since the coefficient maps are (at most) finite-to-one from 3 parameters to 3 coefficients, this means they are all surjective maps, thus their images are the same as they are all equal to $\rr^3$.   One must still check any constraints on the parameters, e.g.~we can assume the parameters must be positive.  Since the coefficients are all positive assuming positive parameters, then their images are the same as they are equal to $\rr_{>0}^3$, and the models are indistinguishable, but \textit{not} via a renaming of the parameters.  

Alternatively, instead of checking (local) injectivity and surjectivity, one could construct the equation system $c(p) = c'(p')$ and attempt to
``solve'' for one set of parameters in terms of the other, and vice versa, using
Gr\"obner basis calculations.  For example, letting $c(p)$ correspond to model $\cm'$ where $p=(a_{02},a_{21},a_{12})$and letting $c'(p')$ correspond to model $\cm''$ where $p'=(b_{01},b_{21},b_{02})$, we obtain $p$ in terms of $p'$:

$$a_{21}=b_{21}, a_{02}=\frac{b_{02}b_{01}+b_{02}b_{21}}{b_{21}}, a_{12}=\frac{-b_{01}b_{02}+b_{01}b_{21}}{b_{21}}$$

and likewise we can solve for $p'$ in terms of $p$:
\begin{multline*}
b_{01} = \frac{1}{2}(a_{02}+a_{12}-a_{21} \pm \sqrt{a_{02}^2+2a_{02}a_{12}+a_{12}^2-2a_{02}a_{21}+2a_{12}a_{21}+a_{21}^2}), \\ 
b_{21}=a_{21}, \\
b_{02} = \frac{1}{2}(a_{02}+a_{12}+a_{21} \mp \sqrt{(-a_{02}-a_{12}-a_{21})^2-4a_{02}a_{21}})
\end{multline*}

See \cite{meshkat-rosen-sullivant} for more examples of this technique.

\end{ex}

\subsection{Godfrey-Chapman geometric rules} \label{subsection:gfgeorules}

Our goal in this paper is to find sufficient conditions for indistinguishability based on graph structure.  Below, we outline previous work that gives necessary conditions for indistinguishability based on graph structure, which we refer to as the Godfrey-Chapman geometric rules \cite{godfrey-chapman}.

\subsubsection{\textit{Rule 1.} The length of the shortest path from any perturbed (input) compartment to any observed (output) compartment is preserved.}


\subsubsection{\textit{Rule 2.} The number of compartments with a path to any given observed compartment is preserved.}

Note that this is satisfied if we assume our model is output connectable (see Definition \ref{defn:outputreachable}).

\subsubsection{Rule 3. The number of compartments that can be reached from a given perturbed compartment is preserved.}

Note that this is satisfied if we assume our model is input connectable (see Definition \ref{defn:outputreachable}).


\subsubsection{Rule 4. The number of traps is preserved.}

A trap is a compartment or a strongly connected set of compartments from which no path exists to any compartment outside the trap, including the environment (i.e.~leaks).  
Note that if we again assume output connectability, then the only possible traps must include an output. 

We will see shortly that the class of graphs we consider automatically satisfy these conditions. 

\begin{rmk} Rule 1 directly connects to the number of coefficients on the right-hand side of the input-output equation.  In the case when there is a single input and single output, the number of nontrivial coefficients on the right-hand side of the corresponding input-output equation according the Corollary 3.4 of \cite{treemodels} is:
\[
\# \text{ on RHS} = \begin{cases}
n-1 & \text{ if } In=Out \\
n-\text{dist}(In,Out) & \text{ if } In \neq Out
\end{cases}.
\]
Here, $\text{dist}(In,Out)$ corresponds to the length of the shortest directed path from $In$ to $Out$.  In the case of multiple inputs or multiple outputs, we simply apply this formula for each distinct input/output pair.
\end{rmk}

\subsection{Permutation indistinguishability and class of graphs}

The previous examples demonstrate the need to find the input-output equations first before testing indistinguishability. 
However, finding the input-output equations and then testing the sufficient conditions requires symbolic computation, which can be tedious for large models.  Our goal in this paper is to circumvent that step of finding the input-output equations and instead determine indistinguishability of certain models from the graph itself.  
We will restrict our efforts to a certain class of models and thus find sufficient conditions for indistinguishability based on the graph structure.  

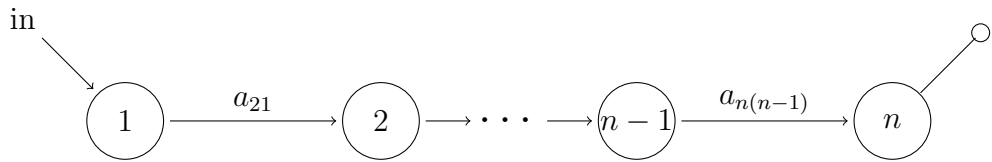
\begin{figure} 
\begin{center}
\begin{tikzpicture}[scale=1.7]
 	\draw (0,0) circle (0.3);
 	\draw (2,0) circle (0.3);
 	\draw (4,0) circle (0.3);
 	\draw (6,0) circle (0.3);

    	\node[] at (0, 0) {1};
    	\node[] at (2, 0) {2};
    	\node[] at (4, 0) {$n-1$};
    	\node[] at (6, 0) {$n$};
	 \draw[->] (0.35, 0) -- (1.65, 0);
	 \draw[->] (4.35,0) -- (5.65,0);
\node at (3,0) {\Large $\cdots$};
\draw[->] (2.35,0) -- (2.7,0);
\draw[->] (3.3,0) -- (3.65,0);

   	 \node[] at (1, 0.15) {$a_{21}$};
	\node[] at (5,.15) {$a_{n(n-1)}$};
	\draw (6.69,.69) circle (0.07);	
	 \draw[-] (6.65, .65 ) -- (6.22, .22);
	 \draw[->] (-.65, .65) -- (-.25, .25);	
   	 \node[] at (-.8,.8) {in};
    	

\end{tikzpicture}

\end{center}
\caption{A basic $n$ compartment path model without any leaks.}
\label{fig:path}
\end{figure} 

We narrow down our class of models to \textit{paths} from input to output with various detours. The basic starting point is a \textit{path model}:

\begin{defn}
    An n-compartment model $\cm = (G, In, Out, Leak)$ is a \textit{path model} if $G$ is given by a path from compartment $1$ to compartment $n$, $In=\{1\}$, and $Out=\{n\}$.  We denote the graph $G$ as $P_n$.
\end{defn}

The graph in Figure \ref{fig:path} gives a path model without any leaks.

Notice that in a path model, the output-reachable subgraph is the entire graph, thus the model is output connectable.  It is also input connectable.  That is, every compartment has a path to the output and every compartment can be reached from the input.  Thus, these models automatically satisfy Rules 2 and 3 of the Godfrey-Chapman necessary conditions.  Rule 1 is automatically satisfied as each model contains a simple path from input to output.  To satisfy Rule 4, we require that compartment $n$ does not contain a leak, which explains why the model $\cm'''$ in Figure \ref{fig:movingleaknonindist} is not indistinguishable with models $\cm, \cm'$, and $\cm''$ in Example \ref{ex:movingleaknonindist}.  

We will also be considering variations on path models that have extra vertices, edges, or leaks. One such example is a path model with a cycle along the last two vertices, which we call a \textit{terminal cycle}, as in model $\cm'$ in Figure \ref{fig:p3leak2ex}.  We call models such as these \textit{skeletal path models}, as the graph contains a ``backbone'' of a path along with other vertices, edges, or leaks:


\begin{defn}  A model $\cm = (G, In, Out, Leak)$ is a \textit{skeletal path model} if $G$ contains a path from compartment $1$ to $n$, $In={1}$, and $Out={n}$.  
\end{defn}

It is clear that Rules 1-4 of the Godfrey-Chapman necessary conditions are satisfied for skeletal path models as long as the shortest path from input to output is the ``skeletal path'' and there is no leak from compartment $n$.  Now we will prove indistinguishability via a simple permutation of parameters: 

\begin{defn} \label{defn:permutation}
    We say two models $\cm$ and $\cm'$ are \textit{indistinguishable via a permutation of parameters} or \textit{permutation indistinguishable} if models $\cm$ and $\cm'$ have the same input-output equations up to a renaming of the parameters.  This means there is a bijection $\Phi$ from the parameters in $\cm$ to the parameters in $\cm'$ such that the coefficients of $\cm'$ are exactly the coefficients of $\cm$ under $\Phi$. 
\end{defn}

\begin{rmk}
    Under the map $\Phi$, we have $p \mapsto p'$.  We will abuse notation and apply $\Phi$ to polynomial functions $f(p)$ as well, where $\Phi(f(p)) = f(\Phi(p))$.
\end{rmk}

We now show models that are permutation indistinguishable are, in fact, indistinguishable.

\begin{prop}
    Assume $\cm$ and $\cm'$ are permutation indistinguishable.  Then they are also indistinguishable.
\end{prop}

\begin{proof} By assumption, $\cm$ and $\cm'$ have the same input-output equations up to a renaming of the parameters.  Thus they have the same input-output equation structure and satisfy the same coefficient dependency relationships.  Let $c(p)$ be the coefficient map for $\cm$ and $c(p')$ be the coefficient map for $\cm'$, where $p'$ is a renaming of the parameters $p$.  Then clearly the image of $c(p)$ equals the image of $c(p')$, and thus the models are indistinguishable. 
\end{proof}


We end this section by noting that permutation indistinguishability is an equivalence relation.

\begin{prop} Permutation indistinguishability is an equivalence relation.
\end{prop}

\begin{proof}
    \begin{itemize}
    \item Reflexive: A model $\cm$ is permutation indistinguishable from itself via the identity map. 
    \item Symmetric: Assume a model $\cm$ is indistinguishable from a model $\cm'$.  This means there is a bijection $\Phi$ from the coefficients of $\cm$ to the coefficients of $\cm'$.  The inverse map $\Phi^{-1}$ proves that $\cm'$ is indistinguishable from $\cm$.
    \item Transitive: Assume $\cm$ is indistinguishable from $\cm'$ and $\cm'$ is indistinguishable from $\cm''$.  This means there is a bijection $\Phi$ from the coefficients $\cm$ to the coefficients of $\cm'$ and also a bijection $\Psi$ from the coefficients $\cm'$ to the coefficients of $\cm''$.  Composing these bijections, we get that there is a bijection $\Psi (\Phi)$ from the coefficients $\cm$ to the coefficients of $\cm''$, thus $\cm$ is indistinguishable from $\cm''$.  
    \end{itemize}
\end{proof}




\section{Paths with leaks and terminal cycles} \label{section:leak}


Recall Figure \ref{fig:permute} and Example \ref{ex:permute}, which we revisit now: 

\begin{ex}\label{ex:p3ind}
Consider the two models $\cm = (P_3,\{1\},\{3\},\{1\})$ and $\cm'=(P_3,\{1\},\{3\},\{2\})$ shown in Figure \ref{fig:permute}. The input-output equations of $\cm$ and $\cm'$ respectively are
\begin{align*}
     \cm: &\quad  \dddot{y}_3+(a_{32}+a_{01}+a_{21})\ddot{y}_3+(a_{01}a_{32}+a_{21}a_{32})\dot{y}_3 = (a_{21}a_{32})u_1\\
    \cm': & \quad \dddot{y}_3+(a_{21}+a_{02}+a_{32})\ddot{y}_3+(a_{02}a_{21}+a_{32}a_{21})\dot{y}_3 = (a_{32}a_{21})u_1\\
\end{align*}

Recall these models are indistinguishable via a renaming of parameters.  We describe this renaming using the map $\Phi$ we defined in Definition \ref{defn:permutation}.  The models $\cm$ and $\cm'$ are indistinguishable under the map 
\[
\Phi \colon
\begin{pmatrix}
a_{01} \\ a_{21} \\ a_{32} 
\end{pmatrix} \mapsto \begin{pmatrix}
a_{02} \\ a_{32} \\ a_{21}
\end{pmatrix}
\]
\end{ex}


We would now like to generalize this example by considering a path model with $n$ vertices.   

The basic path model with no leaks has a compartmental matrix of the form:

\begin{equation} \label{eq:comp-mat}
A = \begin{bmatrix}
-a_{21} & 0 & \cdots & \cdots &  0 & 0 \\
a_{21} & -a_{32} & \ddots & \ddots  & \vdots & \vdots \\
0 & a_{32} & -a_{43} & \ddots &  \vdots & \vdots \\
\vdots & \ddots & \ddots & \ddots &0 & 0 \\
0 & \cdots & 0 & \ddots & -a_{n(n-1)} & 0 \\
0 & \cdots & 0 & 0& a_{n(n-1)} & 0
\end{bmatrix}
\end{equation}

The compartmental matrix of a basic path model with the first row and last column removed has the form

\begin{equation} \label{eq:comp-mat1n}
A^{1,n} = \begin{bmatrix}
a_{21} & -a_{32} & 0 & \cdots  & 0 \\
0 & a_{32} & -a_{43} & \ddots &  \vdots  \\
\vdots & \ddots & \ddots & \ddots &0  \\
0 & \cdots & 0 & \ddots & -a_{n(n-1)}  \\
0 & \cdots & 0 & 0& a_{n(n-1)} 
\end{bmatrix}
\end{equation}
{ which we will take advantage of in the proof of the following theorem.}

\begin{thm} \label{thm:leak}
The path models $\cm_i=(P_n,\{1\},\{n\},\{i\})$ and $\cm_j = (P_n,\{1\},\{n\},\{j\})$ are indistinguishable for all $i,j <n$.
\end{thm}

\begin{proof}

Let $A$ be the compartmental matrix of $\cm_i$, and $B$ be the compartmental matrix of $\cm_j$.  Both $A$ and $B$ have form similar to Equation \ref{eq:comp-mat} with an additional $a_{0i}$ and $a_{0j}$ subtracted from the $i^\text{th}$ and $j^\text{th}$ diagonal elements respectively. The same is true for the matrices $\partial I -A$ and $\partial I -B$.

From Theorem \ref{thm:ioscc}, and specifically Equation \ref{eq:io-det}, the left-hand sides of the input-output equations of $\cm$ and $\cm'$ are generated by the determinant of the $\partial I -A$ and $\partial I -B$ matrices respectively.  Since $\partial I -A$ is a lower triangular matrix, the differential operator of the left-hand side of the input-output equation of $\cm$ is:
\begin{equation}\label{eq:lhs-leakmove}
\left( \frac{d}{dt} + a_{0i} + a_{(i+1)i} \right) \left( \frac{d}{dt} \right)  \prod_{\substack{k \in [n],\\ k \neq i}} \left( \frac{d}{dt} + a_{(k+1)k} \right)
\end{equation}
{ where $[n]=\{1,2,\ldots, n\}.$}
Similarly, the differential operator of the left-hand side of the input-output equation of $\cm'$ is:
\[
\left( \frac{d}{dt} + a_{0j} + a_{(j+1)j} \right) \left( \frac{d}{dt} \right)  \prod_{\substack{k \in [n],\\ k \neq j}} \left( \frac{d}{dt} + a_{(k+1)k} \right).
\]

According to Theorem \ref{thm:ioscc}, and again Equation \ref{eq:io-det}, the right-hand side of the input-output equations of $\cm$ and $\cm'$ are the determinants of $(\partial I - A)^{1,n}$ and $(\partial I - B)^{1,n}$ respectively.  Note that in either case, since $\cm$ and $\cm'$ are identical outside of the leak element, and both $(\partial I - A)^{1,n}$ and $(\partial I - B)^{1,n}$ are upper-triangular, as seen in Equation \ref{eq:comp-mat1n}, with no leak terms on the main diagonal.  Therefore, the right-hand side differential operators of both $\cm$ and $\cm'$ are 
\begin{equation}\label{eq:rhs-leakmove}
(-1)^{1+n}\prod_{k\in [n]} \left(- a_{(k+1)k} \right) = (-1)^{1+n} \cdot (-1)^{n} \prod_{k\in [n]} \left(a_{(k+1)k} \right) = - \prod_{k\in [n]} \left(a_{(k+1)k} \right). 
\end{equation}

Thus, under the [bijective] map $\Phi$ from the parameters of $\cm$ to the parameters of $\cm'$ defined as:
\[
\Phi( a_{uv} ) = \begin{cases}
a_{0j} & \text{ for }u=0\ \& \ v=i \\
a_{(j+1)j} & \text{ for }u=(i+1)\ \&\ v=i \\
a_{(i+1)i}  &  \text{ for }u=(j+1) \ \&\  v=j \\
a_{uv} & \text{ otherwise}
\end{cases}
\]
$\cm$ and $\cm'$ are permutation indistinguishable.

\end{proof}

\begin{figure} 
\begin{center}

\begin{tikzpicture}[scale=.9]
 	\draw (0,0) circle (0.3);
 	\draw (2,0) circle (0.3);
 	\draw (4,0) circle (0.3);
    	\node[] at (0, 0) {1};
    	\node[] at (2, 0) {2};
    	\node[] at (4, 0) {$3$};
	 \draw[->] (0.35, .1) -- (1.65, .1);
	 \draw[->] (2.35,.1) -- (3.65,.1);
   	 \node[] at (1, 0.3) {$a_{21}$};
	\node[] at (3,0.3) {$a_{32}$};
	\draw (4.69,.69) circle (0.07);	
	 \draw[-] (4.65, .65 ) -- (4.22, .22);
	 \draw[->] (-.65, .65) -- (-.25, .25);	
   	 \node[] at (-.8,.8) {in};
	 \draw[->] (2,-.3) -- (2, -.9);	
   	 \node[] at (2.35, -.7) {$a_{02}$};
\draw (-1.5,-2) rectangle (5, 2);
    	\node[] at (2, -1.5) {$\cm$};
    	

\end{tikzpicture}
\begin{tikzpicture}[scale=.9]
 	\draw (0,0) circle (0.3);
 	\draw (2,0) circle (0.3);
 	\draw (4,0) circle (0.3);
    	\node[] at (0, 0) {1};
    	\node[] at (2, 0) {2};
    	\node[] at (4, 0) {$3$};
	 \draw[->] (0.35, .1) -- (1.65, .1);
	 \draw[->] (2.35,.1) -- (3.65,.1);
	 \draw[<-] (2.35,-.1) -- (3.65,-.1);
   	 \node[] at (1, 0.3) {$a_{21}$};
	\node[] at (3,0.3) {$a_{32}$};
	\node[] at (3,-.3) {$a_{23}$};
	\draw (4.69,.69) circle (0.07);	
	 \draw[-] (4.65, .65 ) -- (4.22, .22);
	 \draw[->] (-.65, .65) -- (-.25, .25);	
   	 \node[] at (-.8,.8) {in};
\draw (-1.5,-2) rectangle (5, 2);
    	\node[] at (2, -1.5) {$\cm'$};
    	

\end{tikzpicture}

\end{center}
\caption{Two skeletal path models described in Example \ref{ex:p3leak2ex}}
\label{fig:p3leak2ex}

\end{figure}
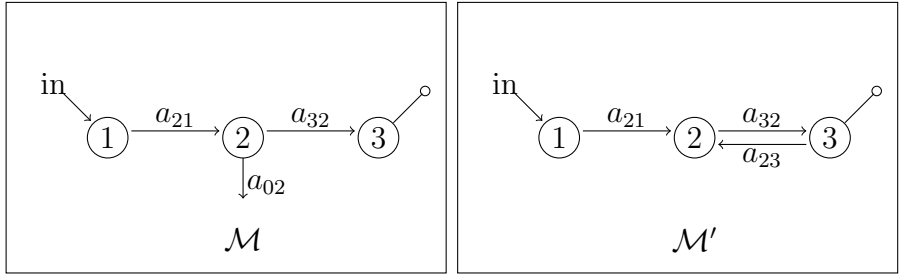

We now again consider Figure \ref{fig:permute} and Example \ref{ex:permute}, which we revisit and rename the models to those in Figure \ref{fig:p3leak2ex}.  

\begin{ex}\label{ex:p3leak2ex}
Consider the two models $\cm = (P_3,\{1\},\{3\},\{2\})$ and $\cm'=(P_3\cup \{3 \to 2\},\{1\},\{3\},\emptyset)$ shown in Figure \ref{fig:p3leak2ex}. 
The input-output equations of $\cm$ and $\cm'$ respectively are
\begin{align*}
    \cm: & \quad \dddot{y}_3+(a_{21}+a_{02}+a_{32})\ddot{y}_3+(a_{02}a_{21}+a_{32}a_{21})\dot{y}_3 = (a_{32}a_{21})u_1\\
    \cm': & \quad \dddot{y}_3+(a_{21}+a_{23}+a_{32})\ddot{y}_3+(a_{23}a_{21}+a_{32}a_{21})\dot{y}_3 = (a_{32}a_{21})u_1
 \end{align*}
Recall these models are indistinguishable via a renaming of parameters.  We describe this renaming using the map $\Phi$ we defined in Definition \ref{defn:permutation}.  The models $\cm$ and $\cm'$ are indistinguishable under the map 
\[
\Phi \colon
\begin{pmatrix}
a_{02} \\ a_{21} \\ a_{32} 
\end{pmatrix} \mapsto \begin{pmatrix}
a_{23} \\ a_{21} \\ a_{32}
\end{pmatrix}
\]
\end{ex}

We would now like to generalize this example.

\begin{thm} \label{thm:terminalcycle}
The path models $\cm=(P_n,\{1\},\{n\},\{n-1\})$ and $\cm' = (P_n \cup \{n \to n-1\},\{1\},\{n\},\emptyset)$ are indistinguishable.
\end{thm}

\begin{proof}
The model $\cm'$ has compartmental matrix of the form:

\begin{equation} \label{eq:comp-mat2}
A' = \begin{bmatrix}
-a_{21} & 0 & \cdots & \cdots &  0 & 0 \\
a_{21} & -a_{32} & \ddots & \ddots  & \vdots & \vdots \\
0 & a_{32} & -a_{43} & \ddots &  \vdots & \vdots \\
\vdots & \ddots & \ddots & \ddots &0 & 0 \\
0 & \cdots & 0 & \ddots & -a_{n(n-1)} & a_{(n-1)n} \\
0 & \cdots & 0 & 0& a_{n(n-1)} & -a_{(n-1)n}
\end{bmatrix}
\end{equation}

Again, via Theorem \ref{thm:ioscc} the left-hand side of the input-output equation corresponding to $\cm'$ has form:
\[
\left[\underbrace{\left( \frac{d}{dt}+a_{(n-1)n}\right)\left(\frac{d}{dt}+a_{n(n-1)} \right)-a_{n(n-1)}a_{(n-1)n}}_{\text{bottom right $2\times 2$ block}} \right] \underbrace{\prod_{i\in [n-1]} \left( \frac{d}{dt} + a_{(k+1)k} \right)}_{\text{first $n-2$ diagonal elements}}
\]

equivalently

\begin{equation}\label{eq:leak2exio}
\left(\frac{d}{dt}+a_{(n-1)n}+a_{n(n-1)}\right) \left(\frac{d}{dt}\right) \prod_{i\in [n-1]} \left( \frac{d}{dt} + a_{(k+1)k} \right)
\end{equation}

Removing the first row and last column of the $A'$ compartmental matrix yields a matrix of the same form as that in Equation \ref{eq:comp-mat1n}.  Therefore, the right-hand side of the input-output equation of $\cm'$ has the same form as Equation \ref{eq:rhs-leakmove}.

Thus, under the [bijective] map $\Phi$ from the parameters of $\cm$ to the parameters of $\cm'$ defined as:
\[
\Phi( a_{uv} ) = \begin{cases}
a_{(n-1)n} & \text{ for }u=0\ \& \ v=(n-1)\\
a_{uv} & \text{ otherwise}
\end{cases}
\]
$\cm$ and $\cm'$ are permutation indistinguishable.

\end{proof}

\section{Detour Indistinguishability} \label{section:detour}

We now consider a scenario as in Figure \ref{fig:detour2}, where a basic path model has a ``detour''{ , i.e. an outgoing edge off of the main path connecting to another connected graph with an incoming edge back to the main path}.  These models still satisfy our conditions of skeletal path models.

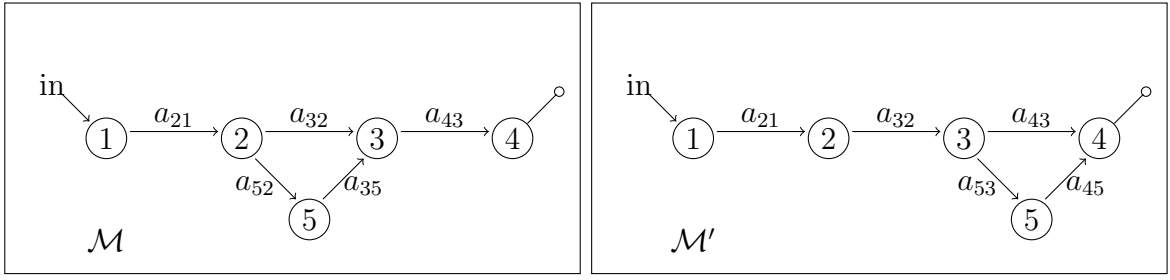
\begin{figure} 
    \centering
\begin{tikzpicture}[scale=.9]
 	\draw (0,0) circle (0.3);
 	\draw (2,0) circle (0.3);
 	\draw (4,0) circle (0.3);
 	\draw (6,0) circle (0.3);
 	\draw (3,-1.2) circle (0.3);
    	\node[] at (0, 0) {1};
    	\node[] at (2, 0) {2};
    	\node[] at (4, 0) {$3$};
    	\node[] at (6, 0) {$4$};
    	\node[] at (3,-1.2) {$5$};
	 \draw[->] (0.35, .1) -- (1.65, .1);
	 \draw[->] (2.35,.1) -- (3.65,.1);
	 \draw[->] (4.35,.1) -- (5.65,.1);
   	 \node[] at (1, 0.3) {$a_{21}$};
	\node[] at (3,0.3) {$a_{32}$};
	\node[] at (5,.3) {$a_{43}$};
	\draw (6.69,.69) circle (0.07);	
	 \draw[-] (6.65, .65 ) -- (6.22, .22);
	 \draw[->] (-.65, .65) -- (-.25, .25);	
   	 \node[] at (-.8,.8) {in};
	 \draw[->] (2.2,-.3) -- (2.8, -.9);	
   	 \node[] at (2.2, -.7) {$a_{52}$};
   	 \draw[->] (3.2,-.9) -- (3.8,-.3);
   	 \node[] at (3.8,-.7) {$a_{35}$};
\draw (-1.5,-2) rectangle (7, 2);
    	\node[] at (0, -1.5) {$\cm$};
    	

\end{tikzpicture}
\begin{tikzpicture}[scale=.9]
 	\draw (0,0) circle (0.3);
 	\draw (2,0) circle (0.3);
 	\draw (4,0) circle (0.3);
 	\draw (6,0) circle (0.3);
 	\draw (5,-1.2) circle (0.3);
    	\node[] at (0, 0) {1};
    	\node[] at (2, 0) {2};
    	\node[] at (4, 0) {$3$};
    	\node[] at (6, 0) {$4$};
    	\node[] at (5,-1.2) {$5$};
	 \draw[->] (0.35, .1) -- (1.65, .1);
	 \draw[->] (2.35,.1) -- (3.65,.1);
	 \draw[->] (4.35,.1) -- (5.65,.1);
   	 \node[] at (1, 0.3) {$a_{21}$};
	\node[] at (3,0.3) {$a_{32}$};
	\node[] at (5,.3) {$a_{43}$};
	\draw (6.69,.69) circle (0.07);	
	 \draw[-] (6.65, .65 ) -- (6.22, .22);
	 \draw[->] (-.65, .65) -- (-.25, .25);	
   	 \node[] at (-.8,.8) {in};
	 \draw[->] (4.2,-.3) -- (4.8, -.9);	
   	 \node[] at (4.2, -.7) {$a_{53}$};
   	 \draw[->] (5.2,-.9) -- (5.8,-.3);
   	 \node[] at (5.8,-.7) {$a_{45}$};
\draw (-1.5,-2) rectangle (7, 2);
    	\node[] at (0, -1.5) {$\cm'$};
    	

\end{tikzpicture}
   \caption{Detour Example} 
    \label{fig:detour2}
\end{figure}

\begin{ex}
Consider the linear compartmental models from Figure \ref{fig:detour2} with compartmental matrices:

\[
A = \begin{bmatrix}
-a_{21} & 0 & 0 & 0 & 0 \\
a_{21} & -a_{32}-a_{52} & 0 & 0 & 0 \\
0 & a_{32} & -a_{43} & 0 & a_{35} \\
0 & 0 & a_{43} & 0 & 0 \\
0 & a_{52} & 0 & 0 & -a_{35}
\end{bmatrix}  \text{ and } B =\begin{bmatrix}
-a_{21} & 0 & 0 & 0 & 0 \\
a_{21} & -a_{32} & 0 & 0 & 0 \\
0 & a_{32} & -a_{43}-a_{53} & 0 & 0 \\
0 & 0 & a_{43} & 0 & a_{45} \\
0 & 0 & a_{53} & 0 & -a_{45}
\end{bmatrix} 
\]








The input-output equations are given by:
\begin{equation*}
    (\partial + a_{21}) (\partial + a_{32}+a_{52})(\partial + a_{43}) \partial (\partial + a_{35}) y_4 
    =(-a_{21}a_{32}a_{43}(\partial + a_{35}) - a_{21}a_{35}a_{43}a_{52}) u_1
\end{equation*}

and
\begin{equation*}
(\partial + a_{21}) (\partial + a_{32})(\partial + a_{43}+a_{53}) \partial (\partial + a_{45})y_4 =
(-a_{21}a_{32}a_{43}(\partial + a_{45}) - a_{32}a_{45}a_{21}a_{53}) u_1 
\end{equation*}


These two models are indistinguishable via the following map:

\[
\begin{pmatrix}
a_{21} \\ a_{32} \\ a_{43} \\ a_{35} \\ a_{52}
\end{pmatrix} \mapsto 
\begin{pmatrix}
a_{32} \\ a_{43} \\ a_{21} \\ a_{45} \\ a_{53}
\end{pmatrix}
\]

\end{ex}

\begin{figure} 
	\begin{tikzpicture}[scale=.9]
 	\draw (0,0) circle (0.3);
 	\draw (2,0) circle (0.3);
 	\draw (4,0) circle (0.3);
 	\draw (6,0) circle (0.3);
    	\node[] at (0, 0) {1};
    	\node[] at (2, 0) {2};
    	\node[] at (4, 0) {$3$};
    	\node[] at (6, 0) {$4$};
	 \draw[->] (0.35, .1) -- (1.65, .1);
	 \draw[->] (2.35,.1) -- (3.65,.1);
	 \draw[->] (4.35,.1) -- (5.65,.1);
   	 \node[] at (1, 0.3) {$a_{21}$};
	\node[] at (3,0.3) {$a_{32}$};
	\node[] at (5,.3) {$a_{43}$};
	\draw (6.69,.69) circle (0.07);	
	 \draw[-] (6.65, .65 ) -- (6.22, .22);
	 \draw[->] (-.65, .65) -- (-.25, .25);	
   	 \node[] at (-.8,.8) {in};

\draw[->] (.25,.25) -- (1.8,1.8);
\draw[->] (2.2,1.8) -- (3.8,.25);
\node[] at (2,2) {$D$};
\node[] at (.7,1.2) {$a_{s1}$};
\node[] at (3.3,1.2) {$a_{3t}$};

\node[] at (5,2) {$\cm$};


 	\draw (0,-2) circle (0.3);
 	\draw (2,-2) circle (0.3);
 	\draw (4,-2) circle (0.3);
 	\draw (6,-2) circle (0.3);
    	\node[] at (0, -2) {1};
    	\node[] at (2, -2) {2};
    	\node[] at (4, -2) {$3$};
    	\node[] at (6, -2) {$4$};
	 \draw[->] (0.35, -1.9) -- (1.65, -1.9);
	 \draw[->] (2.35,-1.9) -- (3.65,-1.9);
	 \draw[->] (4.35,-1.9) -- (5.65,-1.9);
   	 \node[] at (1, -1.7) {$a_{21}$};
	\node[] at (3,-1.7) {$a_{32}$};
	\node[] at (5,-1.7) {$a_{43}$};
	\draw (6.69,-1.31) circle (0.07);	
	 \draw[-] (6.65, -1.35 ) -- (6.22, -1.78);
	 \draw[->] (-.65, -1.35) -- (-.25, -1.75);	
   	 \node[] at (-.8,-1.2) {in};

\draw[->] (2.25,-2.25) -- (3.8,-3.8);
\draw[->] (4.2,-3.8) -- (5.8,-2.25);
\node[] at (4,-4) {$D$};
\node[] at (2.7,-3.2) {$a_{s2}$};
\node[] at (5.3,-3.2) {$a_{4t}$};

\node[] at (1,-4) {$\cm'$};

\draw[->,red,thick,dashed] (1,0) -- (2.7,-1.5);
\draw[->,red,thick,dashed] (3,0) -- (4.7,-1.5);
\draw[->,red,thick,dashed] (5,0) -- (1.3,-1.5);
\node[] at (1,-.75) {$\color{red}\Phi$};

\draw (-1.5,-5) rectangle (7, 3);

\end{tikzpicture}
\caption{An example map $\Phi$ for two path models with detours}
\label{fig:detour1}
\end{figure}
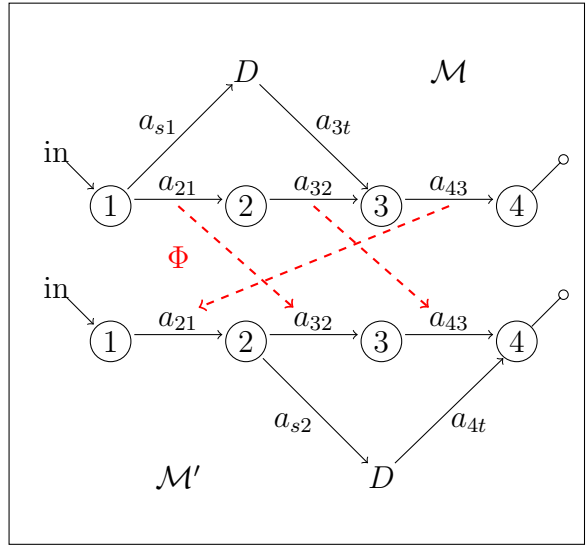

{  We now generalize our result, as in Figure \ref{fig:detour1}.}

\begin{thm} \label{thm:detour}
Let $P_n$ be the directed path graph from $1$ to $n$, and let $D$ be some connected directed graph.  Then, for some $1 \leq i\ {  \leq }\ j \leq n-1$ { with $i<n-1$} and $s,t \in V(D)$, the \textit{detour model} 
\[\cm = (P_n \cup D \cup \{i \to s\} \cup \{t \to j\} , \{1\},\{n\},Leak)\] 
is indistinguishable to the model 
\[\cm' = (P_n \cup D \cup \{(i+1) \to s\} \cup \{t \to (j+1)\} , \{1\},\{n\},Leak)\]
\end{thm}


\begin{proof}

We claim the following map between the coefficients of $\cm$ and $\cm'$ yields indistinguishability between these two models:

\[
\Phi(a_{uv}) = \begin{cases}
a_{(u+1)(v+1)}, \text{ if } u=v+1 \text{ and } u < n & (\text{edges along $P_n$ not including last edge})\\
a_{21}, \text{ if } u=n \text{ and } v=n-1& (\text{last edge on $P_n$}) \\
a_{s(i+1)}, 
\text{ if } u = s \text{ and } v = i & (\text{off-ramp off of path}) \\
a_{(j+1)t}, 
\text{ if } u = j \text{ and } v = t & (\text{on-ramp back to path}) \\
a_{uv}, \text{ otherwise }  & (\text{on detour}) 
\end{cases}
\]



First, we construct the compartmental matrix of the detour model, which we call $A$, using the compartment matrices of the elements of the detour model.  Let $A_D$ be the compartmental matrix of the underlying detour with an additional $a_{jt}$ in the $j^{\text{th}}$ diagonal sum.  Then, the compartmental matrix of the detour model $\cm$ can be generated as a block matrix as follows:
\[
\partial I - A = \left[\begin{array}{cccccc|c}
      \partial + a_{21}  &0 &\cdots &\cdots & 0 & 0& \mathbf{0} \\
      -a_{21}  & \ddots&\ddots &&\vdots&\vdots &\mathbf{0}  \\
      0  & \ddots & \partial + a_{si}+a_{(i+1)i} & \ddots &\vdots &\vdots &\mathbf{0} \\
      \vdots  &\ddots &\ddots  & \ddots & 0 & 0& -a_{jt}\\
      \vdots  & & \ddots  & \ddots & \partial + a_{n(n-1)} & 0 & \mathbf{0} \\
      0    & \cdots & \cdots & 0 & -a_{n(n-1)} & \partial &\mathbf{0}\\
     \hline
     \mathbf{0}  & \mathbf{0} & -a_{si} & \mathbf{0} & \mathbf{0} &\mathbf{0} & \partial I -  A_D
\end{array}
\right]
\]
where the top-right block only contains the element $a_{jt}$ (in row $j$ and column $t$) corresponding to the edge from the detour back onto the path, and similarly the bottom-left block only contains the element $a_{si}$ (in row $s$ and column $i$) corresponding to the edge from the path to the detour, with all other entries being zero.

The second detour model $\cm'$ has compartmental matrix of the form:
\[
\partial I - A' = \left[\begin{array}{cccccc|c}
      \partial + a_{21}  &0 &\cdots &\cdots & 0 & 0& \mathbf{0} \\
      -a_{21}  & \ddots&\ddots &&\vdots&\vdots &\mathbf{0}  \\
      0  & \ddots & \partial + a_{s(i+1)}+a_{(i+2)(i+1)} & \ddots &\vdots &\vdots & \mathbf{0} \\
      \vdots  &\ddots &\ddots  & \ddots & 0 & 0& -a_{(j+1)t}\\
      \vdots  & & \ddots  & \ddots & \partial + a_{n(n-1)} & 0 & \mathbf{0} \\
      0    & \cdots & \cdots & 0 & -a_{n(n-1)} & \partial  &\mathbf{0}\\
     \hline
     \mathbf{0}  & \mathbf{0} & -a_{s(i+1)} & \mathbf{0} & \mathbf{0} &\mathbf{0} &  \partial I - A_D'
\end{array}
\right]
\]
According to Equation \ref{eq:io-det}, the coefficients of the left-hand side of the input-output equation of $\cm$ and $\cm'$ come from the determinants $\det(\partial I - A)$ and $\det(\partial I - A')$ respectively. Note that in both determinants, we can expand along the $n^{\text{th}}$ (the last column in the upper-left block), which only contains a $\partial$ yielding:

\begin{equation}\label{eq:lhsA}
\det(\partial I - A) =\partial \det\left( \begin{array}{ccccc|c}
      \partial + a_{21}  &0 &\cdots &\cdots & 0 & \mathbf{0} \\
      -a_{21}  & \ddots&\ddots &&\vdots &\mathbf{0}  \\
      0  & \ddots & \partial + a_{si}+a_{(i+1)i} & \ddots &\vdots  &\mathbf{0} \\
      \vdots  &\ddots &\ddots & \ddots & 0 &  -a_{jt}\\
      0  &\cdots & 0  & -a_{(n-1)(n-2)} & \partial + a_{n(n-1)} & \mathbf{0} \\
     \hline
     \mathbf{0}  & \mathbf{0} & -a_{si} & \mathbf{0} & \mathbf{0}  & \partial I -  A_D
\end{array}
\right)
\end{equation}
and 
\begin{equation}\label{eq:lhsAprime}
\det(\partial I - A') =\partial \det\left( \begin{array}{ccccc|c}
      \partial + a_{21}  &0 &\cdots &\cdots & 0 & \mathbf{0} \\
      -a_{21}  & \ddots&\ddots &&\vdots &\mathbf{0}  \\
      0  & \ddots & \partial + a_{s(i+1)}+a_{(i+2)(i+1)} & \ddots &\vdots  &\mathbf{0} \\
      \vdots  &\ddots &\ddots  & \ddots & 0 &  -a_{(j+1)t}\\
      0  &\cdots & 0  & -a_{(n-1)(n-2)} & \partial + a_{n(n-1)} & \mathbf{0} \\
     \hline
     \mathbf{0}  & \mathbf{0} & -a_{s(i+1)} & \mathbf{0} & \mathbf{0}  & \partial I -  A_D'
\end{array}
\right).
\end{equation}

Expanding down the $n-1^{\text{st}}$ column of the determinant in Equation \ref{eq:lhsA} and the first row in Equation \ref{eq:lhsAprime}, we get
\begin{equation}\label{eq:lhsAshifted2}
\det(\partial I - A) =\partial(\partial + a_{n(n-1)}) \det\left( \begin{array}{ccccc|c}
       \partial +a_{21} &0 &\cdots & \cdots & 0 & \mathbf{0} \\
      -a_{21} & \ddots&\ddots &  & \vdots  &\mathbf{0}  \\
      0  & \ddots & \partial + a_{si}+a_{(i+1)i} & \ddots &\vdots  &\mathbf{0} \\
      \vdots  &\ddots &\ddots  & \ddots & 0 &  -a_{jt}\\
      0  &\cdots & 0  & -a_{(n-2)(n-3)} & \partial + a_{(n-1)(n-2)} & \mathbf{0} \\
     \hline
     \mathbf{0}  & \mathbf{0} & -a_{si} & \mathbf{0} & \mathbf{0}  & \partial I -  A_D
\end{array}
\right)
\end{equation}
and
\begin{equation}\label{eq:lhsAprime2}
\det(\partial I - A') =\partial(\partial + a_{21}) \det\left( \begin{array}{ccccc|c}
       \partial +a_{32} &0 &\cdots &\cdots & 0 & \mathbf{0} \\
      -a_{32}  & \ddots&\ddots &&\vdots &\mathbf{0}  \\
      0  & \ddots & \partial + a_{s(i+1)}+a_{(i+2)(i+1)} & \ddots &\vdots  &\mathbf{0} \\
      \vdots  &\ddots &\ddots  & \ddots & 0 &  -a_{(j+1)t}\\
      0  &\cdots & 0  & -a_{(n-1)(n-2)} & \partial + a_{n(n-1)} & \mathbf{0} \\
     \hline
     \mathbf{0}  & \mathbf{0} & -a_{s(i+1)} & \mathbf{0} & \mathbf{0}  & \partial I -  A_D'
\end{array}
\right).
\end{equation}

 Note that under the map $\Phi$, the entries of the matrix in the determinant in Equation \ref{eq:lhsAshifted2} are exactly the entries of matrix in the determinant in Equation \ref{eq:lhsAprime2}.  Also, since $\Phi(a_{n(n-1)})=a_{21}$, then $\Phi$ applied to the coefficients of $\det(\partial I -A))$ gives $\det(\partial I - A')$, meaning the coefficients of the derivatives of $y_n$ are indistinguishable under $\Phi$.


Now we consider the coefficients of the derivatives of $u_1$ which, again according to Equation \ref{eq:io-det}, come from the determinants of $(\partial I - A)^{(1,n)}$ and $(\partial I - A')^{(1,n)}$, respectively.

We split our argument into four cases: 
\begin{itemize}
    \item $i=1$ and $j>1$
    \item $i=1$ and $j=1$
    \item $i=2$
    \item $i \geq 3$
\end{itemize}
For the case where $i \geq 3$ we have:

\[
(\partial I - A)^{1,n} = \left[\begin{array}{ccccc|c}
      -a_{21}  & \partial + a_{32} &0 &\cdots &0 &\mathbf{0}  \\
      0  & \ddots & \partial + a_{si}+a_{(i+1)i} & \ddots &\vdots  &\mathbf{0} \\
      \vdots  &\ddots &\ddots  & \ddots & 0 &  -a_{jt}\\
      \vdots  & & \ddots  & \ddots & \partial + a_{n(n-1)} &  \mathbf{0} \\
      0    & \cdots & \cdots & 0 & -a_{n(n-1)}  &\mathbf{0}\\
     \hline
     \mathbf{0}  & \mathbf{0} & -a_{si} & \mathbf{0} & \mathbf{0} & \partial I -  A_D
\end{array}
\right]
\]

and

\[
(\partial I - A')^{1,n} = \left[\begin{array}{ccccc|c}
      -a_{21}  & \partial + a_{32} &0 & \cdots & 0 &\mathbf{0}  \\
      0  & \ddots & \partial + a_{s(i+1)}+a_{(i+2)(i+1)} & \ddots &\vdots  &\mathbf{0} \\
      \vdots  &\ddots & \ddots & \ddots & 0 &  -a_{(j+1)t}\\
      \vdots  & & \ddots  & \ddots & \partial + a_{n(n-1)} & \mathbf{0} \\
      0    & \cdots & \cdots & 0 & -a_{n(n-1)}   &\mathbf{0}\\
     \hline
     \mathbf{0}  & \mathbf{0} & -a_{s(i+1)} & \mathbf{0} & \mathbf{0}  &  \partial I - A_D'
\end{array}
\right]
\]
Expanding along the $n-1^{\text{st}}$ row (last row in the top-left block) in $(\partial I - A)^{1,n}$ and the first column in $(\partial I - A')^{1,n}$ we get

\[
\det(\partial I - A)^{1,n} = -a_{n(n-1)} \det\left(\begin{array}{ccccc|c}
      -a_{21}  & \partial + a_{32} &0 &\cdots &0 &\mathbf{0}  \\
      0  & \ddots & \partial + a_{si}+a_{(i+1)i} & \ddots &\vdots  &\mathbf{0} \\
      \vdots  &\ddots & \ddots  & \ddots & 0 &  -a_{jt}\\
      \vdots  & & \ddots  & \ddots & \partial + a_{(n-1)(n-2)} &  \mathbf{0} \\
      0    & \cdots & \cdots & 0 & -a_{(n-1)(n-2)}  &\mathbf{0}\\
     \hline
     \mathbf{0}  & \mathbf{0} & -a_{si} & \mathbf{0} & \mathbf{0} & \partial I -  A_D
\end{array}
\right)
\]

and

\[
\det(\partial I - A')^{1,n} = -a_{21}\det\left(\begin{array}{ccccc|c}
      -a_{32}  & \partial + a_{43} &0 & \cdots & 0 &\mathbf{0}  \\
      0  & \ddots & \partial + a_{s(i+1)}+a_{(i+2)(i+1)} & \ddots &\vdots  &\mathbf{0} \\
      \vdots  &\ddots &\ddots  & \ddots & 0 &  -a_{(j+1)t}\\
      \vdots  & & \ddots  & \ddots & \partial + a_{n(n-1)} & \mathbf{0} \\
      0    & \cdots & \cdots & 0 & -a_{n(n-1)}   &\mathbf{0}\\
     \hline
     \mathbf{0}  & \mathbf{0} & -a_{s(i+1)} & \mathbf{0} & \mathbf{0}  &  \partial I - A_D'
\end{array}
\right).
\]
Thus, by the same argument as for the coefficients of the derivatives of $y_n$, the coefficients of $u_1$ are the same up to parameter permutation, as described by $\Phi$. 

For the case $i=1$ { and $j>1$}, we can still expand along the $n-1^{st}$ row (last row in the top-left block) in $(\partial I - A)^{1,n}$ and the first column in $(\partial I - A')^{1,n}$, but our matrices have the form:

\[
\det(\partial I - A)^{1,n} = -a_{n(n-1)} \det\left(\begin{array}{ccccc|c}
      -a_{21}  & \partial + a_{32} &0 &\cdots &0 &\mathbf{0}  \\
      0  & \ddots & \partial + a_{43} & \ddots &\vdots  &\mathbf{0} \\
      \vdots  &\ddots & \ddots  & \ddots & 0 &  -a_{jt}\\
      \vdots  & & \ddots  & \ddots & \partial + a_{(n-1)(n-2)} &  \mathbf{0} \\
      0    & \cdots & \cdots & 0 & -a_{(n-1)(n-2)}  &\mathbf{0}\\
     \hline
     -a_{s1}  & \mathbf{0} & \mathbf{0} & \mathbf{0} & \mathbf{0} & \partial I -  A_D
\end{array}
\right)
\]

and

\[
\det(\partial I - A')^{1,n} = -a_{21}\det\left(\begin{array}{ccccc|c}
      -a_{32}  & \partial + a_{43} &0 & \cdots & 0 &\mathbf{0}  \\
      0  & \ddots & \partial + a_{54} & \ddots &\vdots  &\mathbf{0} \\
      \vdots  &\ddots &\ddots  & \ddots & 0 &  -a_{(j+1)t}\\
      \vdots  & & \ddots  & \ddots & \partial + a_{n(n-1)} & \mathbf{0} \\
      0    & \cdots & \cdots & 0 & -a_{n(n-1)}   &\mathbf{0}\\
     \hline
     -a_{s2}  & \mathbf{0} & \mathbf{0} & \mathbf{0} & \mathbf{0}  &  \partial I - A_D'
\end{array}
\right).
\]
Thus, the coefficients of $u_1$ are the same up to parameter permutation, as described by $\Phi$.

{ If $i=j=1$, we can still expand along the $n-1^{st}$ row (last row in the top-left block) in $(\partial I - A)^{1,n}$ and the first column in $(\partial I - A')^{1,n}$, but our matrices have the form:

\[
\det(\partial I - A)^{1,n} = -a_{n(n-1)} \det\left(\begin{array}{ccccc|c}
      -a_{21}  & \partial + a_{32} &0 &\cdots &0 &\mathbf{0}  \\
      0  & \ddots & \partial + a_{43} & \ddots &\vdots  &\mathbf{0} \\
      \vdots  &\ddots & \ddots  & \ddots & 0 &  \mathbf{0}\\
      \vdots  & & \ddots  & \ddots & \partial + a_{(n-1)(n-2)} &  \mathbf{0} \\
      0    & \cdots & \cdots & 0 & -a_{(n-1)(n-2)}  &\mathbf{0}\\
     \hline
     -a_{s1}  & \mathbf{0} & \mathbf{0} & \mathbf{0} & \mathbf{0} & \partial I -  A_D
\end{array}
\right)
\]

and

\[
\det(\partial I - A')^{1,n} = -a_{21}\det\left(\begin{array}{ccccc|c}
      -a_{32}  & \partial + a_{43} &0 & \cdots & 0 & \mathbf{0}  \\
      0  & \ddots & \partial + a_{54} & \ddots &\vdots  &\mathbf{0} \\
      \vdots  &\ddots &\ddots  & \ddots & 0 &  \mathbf{0}\\
      \vdots  & & \ddots  & \ddots & \partial + a_{n(n-1)} & \mathbf{0} \\
      0    & \cdots & \cdots & 0 & -a_{n(n-1)}   &\mathbf{0}\\
     \hline
     -a_{s2}  & \mathbf{0} & \mathbf{0} & \mathbf{0} & \mathbf{0}  &  \partial I - A_D'
\end{array}
\right).
\]
Thus, the coefficients of $u_1$ are the same up to parameter permutation, as described by $\Phi$.
}

Likewise, for the case $i=2$, we can still expand along the $n-1^{st}$ row (last row in the top-left block) in $(\partial I - A)^{1,n}$ and the first column in $(\partial I - A')^{1,n}$, but our matrices look like:

\[
\det(\partial I - A)^{1,n} = -a_{n(n-1)} \det\left(\begin{array}{ccccc|c}
      -a_{21}  & \partial + a_{32} + a_{s2} &0 &\cdots &0 &\mathbf{0}  \\
      0  & -a_{32} & \partial + a_{43}  &\ddots   &\vdots & \mathbf{0} \\
      \vdots  &\ddots & \ddots  & \ddots & 0 &  -a_{jt}\\
      \vdots  & & \ddots  & \ddots & \partial + a_{(n-1)(n-2)} &  \mathbf{0} \\
      0    & \cdots & \cdots & 0 & -a_{(n-1)(n-2)}  &\mathbf{0}\\
     \hline
     \mathbf{0}  & -a_{s2} & \mathbf{0} & \mathbf{0} & \mathbf{0} & \partial I -  A_D
\end{array}
\right)
\]

and

\[
\det(\partial I - A')^{1,n} = -a_{21}\det\left(\begin{array}{ccccc|c}
      -a_{32}  & \partial + a_{43} + a_{s3} &0 & \cdots & 0 &\mathbf{0}  \\
      0  & -a_{43} & \partial + a_{54} & \ddots &\vdots  &\mathbf{0} \\
      \vdots  &\ddots &\ddots  & \ddots & 0 &  -a_{(j+1)t}\\
      \vdots  & & \ddots  & \ddots & \partial + a_{n(n-1)} & \mathbf{0} \\
      0    & \cdots & \cdots & 0 & -a_{n(n-1)}   &\mathbf{0}\\
     \hline
     \mathbf{0}  & -a_{s3} & \mathbf{0} & \mathbf{0} & \mathbf{0}  &  \partial I - A_D'
\end{array}
\right).
\]

Thus, the coefficients of $u_1$ are the same up to parameter permutation, as described by $\Phi$.


\end{proof}

{ 
\begin{rmk}
    Note that in the previous proof, $i<n-1$ is clearly necessary if $i<j$, since if $i=n-1$, then $j=n$ and $j+1=n+1$ would be a compartment not on the skeletal path.  In the case when $i=j$, one could imagine that $i=j=n-1$ (an exchange in the $n-1$ node on the skeletal path), however the expansion of $\partial I - A'$ in Equation \ref{eq:lhsAprime} down the $n^{\text{th}}$ column no longer contains only one non-zero element.  Thus, $\partial I -A$ has a different form than $\partial I -A'$, resulting in two models which are \textbf{not} indistinguishable.
\end{rmk}

}




\section{Sink Models} \label{section:sink}

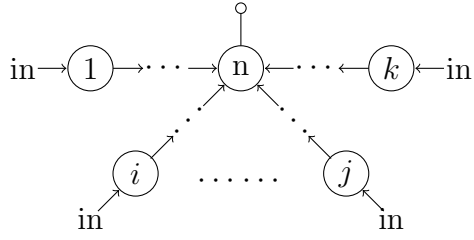
\begin{figure}
    \centering
    \begin{tikzpicture}[scale=1]
 	\draw (0,0) circle (0.3);
 	\draw (-2,0) circle (0.3);
 	\draw (2,0) circle (0.3);
 	\draw (-1.4,-1.4) circle (0.3);
 	\draw (1.4,-1.4) circle (0.3);
 	\node[] at (-.3, -1.5) {$\cdots$};
 	\node[] at (.3, -1.5) {$\cdots$};

    	\node[] at (0, 0) {n};
    	\node[] at (-2, 0) {1};
    	\node[] at (-1.4, -1.4) {$i$};
    	\node[] at (1.4, -1.4) {$j$};
    	\node[] at (2,0) {$k$};
	 \draw[->] (-1.7, 0) -- (-1.3, 0);
	 \node[] at (-1,0) {$\cdots$};
	 \draw[->] (-.7,0) -- (-.32,0);
	 \draw[->] (1.7, 0) -- (1.3, 0);
	 \node[] at (1,0) {$\cdots$};
	 \draw[->] (.7,0) -- (.32,0);
    \draw[->] (-1.2,-1.2) -- (-.9,-.9);
    \node at (-.85,-.85) {$\cdot$};
    \node at (-.7,-.7) {$\cdot$};
    \node at (-.55,-.55) {$\cdot$};
    \draw[->] (-.5,-.5) -- (-.22,-.22);
    \draw[->] (1.2,-1.2) -- (.9,-.9);
    \node at (.85,-.85) {$\cdot$};
    \node at (.7,-.7) {$\cdot$};
    \node at (.55,-.55) {$\cdot$};
    \draw[->] (.5,-.5) -- (.22,-.22);

	\draw (0,.8) circle (0.07);	
	 \draw[-] (0, .3 ) -- (0, .73);
	 \draw[->] (-2.7, 0) -- (-2.32, 0);	
   	 \node[] at (-2.9,0) {in};
	 \draw[->] (2.7, 0) -- (2.32, 0);	
   	 \node[] at (2.9,0) {in};
   	 
   	 \draw[->] (-1.9,-1.9) -- (-1.62,-1.62);
   	 \node at (-2,-2) {in};
   	 \draw[->] (1.9,-1.9) -- (1.62,-1.62);
   	 \node at (2,-2) {in};

    	

\end{tikzpicture}
    \caption{Basic sink model.}
    \label{fig:basicsink}
\end{figure}

We now extend the idea of indistinguishability of path models to a family of models called \textit{sink models}.  A sink model is defined as a combination of skeletal path models which all of have same output, as depicted in Figure \ref{fig:basicsink}.

\begin{cor} \label{cor:sink}
Suppose two sink models $\cm$ and $\cm'$ are identical outside of one path on each model which on their own are permutation indistinguishable, then $\cm$ and $\cm'$ are permutation indistinguishable.
\end{cor}

\begin{proof}
Let $A$ and $A'$ be the compartmental matrices corresponding to $\cm$ and $\cm'$ respectively.  Let $A_1$ be the compartmental matrix of the path model in $\cm$ which differs from the corresponding path in $\cm'$, but is indistinguishable.  Let $A_1'$ be the corresponding compartmental matrix in $\cm'$.
\[
A = \left[
\begin{array}{ccccc}
    A_1 & * & * & * & * \\
    0 & A_2 & 0 & \cdots  & 0 \\
    0 & 0 & A_3 & \ddots & \vdots  \\
    \vdots  &  & \ddots  & \ddots  & 0\\
    0 & \cdots  & \cdots & 0& A_k
\end{array}
\right]
\]
Note that each $*$ in the matrix above corresponds to a single $a_{ni}$ in the $n^{\text{th}}$ row of $A$ corresponding to the edge which connects each of the paths corresponding to $A_2,A_3,\ldots , A_k$ to the output, $n$.

Similarly, the compartmental matrix corresponding to $\cm'$ is

\[
A' = \left[
\begin{array}{ccccc}
    A_1' & * & * & * & * \\
    0 & A_2' & 0 & \cdots & 0 \\
    0 & 0 & A_3' & \ddots & \vdots  \\
    \vdots  &  & \ddots  & \ddots  & 0\\
    0 & \cdots  & \cdots & 0& A_k'
\end{array} \right]
=
\left[
\begin{array}{ccccc}
    A_1 '& * & * & * & * \\
    0 & A_2 & 0 & \cdots & 0 \\
    0 & 0 & A_3 & \ddots & \vdots  \\
    \vdots  &  & \ddots  & \ddots  & 0\\
    0 & \cdots  & \cdots & 0& A_k
\end{array}
\right]
\]
Note that each of the paths not corresponding to $A_1$ or $A_1'$ in $\cm$ and $\cm'$ respectively are exactly the same, hence $A_i=A_i'$ for each of the $i\geq 2$.

Thus, the coefficients of the left-hand side of the all $k$ of the input-output equations of $\cm$ and $\cm'$ respectively are given by
\[\det(\partial I - A) = \det(\partial I - A_1) \det(\partial I - A_2) \cdots \det(\partial I - A_k)  \]
and 
\begin{align*}
\det(\partial I - A') &= \det(\partial I - A_1') \det(\partial I - A_2') \cdots \det(\partial I - A_k') 
\\
&=\det(\partial I - A_1') \det(\partial I - A_2) \cdots \det(\partial I - A_k) 
\end{align*}

Since the models corresponding to $A_1$ and $A_1'$ are indistinguishable, then $\det(\partial I - A_1)$ and $\det(\partial I - A_1')$ are identical up to permutation of parameters.  Thus, $\det(\partial I - A)$ and $\det(\partial I - A')$ are identical up to permutation of parameters.

Similarly, the coefficients for each of the derivatives of $u_{j}$ on the right-hand side of each of the input-output equations of $\cm$ and $\cm'$ is generated by 
\begin{align*}
\det(\partial I - A)^{j,n}.
\end{align*}
For $j=1$ corresponding to the input in the path model in $\cm$ which is different but indistinguishable from the path model in $\cm'$, we get 
\begin{align*}
\det(\partial I - A)^{1,n} &= \det(\partial I - A_1)^{1,n} \det(\partial I - A_2) \cdots \det(\partial I - A_k) 
\end{align*}
since the first row and $n^{\text{th}}$ column are both in $A_1$.
Again, for the $j=1$ coefficient on the right-hand side of $\cm'$ we get
\begin{align*}
\det(\partial I - A')^{(1,n)} &= \det(\partial I - A_1')^{1,n} \det(\partial I - A_2) \cdots \det(\partial I - A_k) 
\end{align*}
Similar to the argument for the left-hand side coefficients, since the models corresponding to $A_1$ and $A_1'$ are indistinguishable, then $\det(\partial I - A_1)^{(1,n)}$ and $\det(\partial I - A_1')$ are identical up to permutation of parameters.  Thus, $\det(\partial I - A)$ and $\det(\partial I - A')$ are identical up to permutation of parameters.

\end{proof}

\begin{ex} 
Sink model examples from Figure \ref{fig:sink}. 
Consider the linear compartmental models with compartmental matrices:
\[
A = \begin{bmatrix}
-a_{21}-a_{41} & 0 & 0 & 0 & 0 \\
a_{21} & -a_{32} & 0 & a_{24} & 0 \\
0 & a_{32} & 0 & 0 & a_{35} \\
a_{41} & 0 & 0 & -a_{24} & 0 \\
0 & 0 & 0 & 0 & -a_{35}
\end{bmatrix}  \text{ and } A' =\begin{bmatrix}
-a_{21} & 0 & 0 & 0 & 0 \\
a_{21} & -a_{32}-a_{42} & 0 & 0 & 0 \\
0 & a_{32} & 0 & a_{34} & a_{35} \\
0 & a_{42} & 0 & -a_{34} & 0 \\
0 & 0 & 0 & 0 & -a_{35}
\end{bmatrix} 
\]

The input-output equations are given by:
\begin{multline*}
(\partial + a_{21} + a_{41}) (\partial + a_{32})\partial(\partial + a_{24})  (\partial + a_{35}) y_3 = \\
(-a_{21}a_{32}(\partial + a_{24})(\partial + a_{35})+a_{41}a_{35}a_{32}a_{24})u_1 + (\partial + a_{21} + a_{41})(\partial + a_{32})a_{35}a_{24}u_5
\end{multline*}

and
\begin{multline*}
(\partial + a_{21}) (\partial + a_{32} + a_{42})\partial(\partial + a_{34})  (\partial + a_{35}) y_3 = \\
(-a_{21}a_{32}(\partial + a_{34})(\partial + a_{35})+a_{42}a_{35}a_{21}a_{34})u_1 + (\partial + a_{21})(\partial + a_{32} + a_{42})a_{35}a_{34}u_5
\end{multline*}

\end{ex}

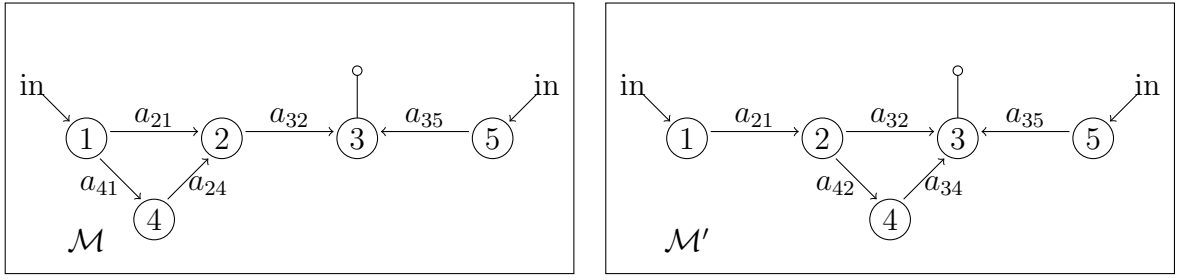
\begin{figure} 
    \centering
\begin{tikzpicture}[scale=.9]
 	\draw (0,0) circle (0.3);
 	\draw (2,0) circle (0.3);
 	\draw (4,0) circle (0.3);
 	\draw (6,0) circle (0.3);
 	\draw (1,-1.2) circle (0.3);
    	\node[] at (0, 0) {1};
    	\node[] at (2, 0) {2};
    	\node[] at (4, 0) {$3$};
    	\node[] at (6, 0) {$5$};
    	\node[] at (1,-1.2) {$4$};
	 \draw[->] (0.35, .1) -- (1.65, .1);
	 \draw[->] (2.35,.1) -- (3.65,.1);
	 \draw[<-] (4.35,.1) -- (5.65,.1);
   	 \node[] at (1, 0.3) {$a_{21}$};
	\node[] at (3,0.3) {$a_{32}$};
	\node[] at (5,.3) {$a_{35}$};
	\draw (4,1) circle (0.07);	
	 \draw[-] (4, .3 ) -- (4, .93);
	 \draw[->] (6.65, .65) -- (6.25, .25);	
   	 \node[] at (6.8,.8) {in};
	 \draw[->] (-.65, .65) -- (-.25, .25);	
   	 \node[] at (-.8,.8) {in};
   	 
	 \draw[->] (.2,-.3) -- (.8, -.9);	
   	 \node[] at (.2, -.7) {$a_{41}$};
   	 \draw[->] (1.2,-.9) -- (1.8,-.3);
   	 \node[] at (1.8,-.7) {$a_{24}$};
\draw (-1.2,-2) rectangle (7.2, 2);
    	\node[] at (0, -1.5) {$\cm$};
    	

\end{tikzpicture}$\quad$\begin{tikzpicture}[scale=.9]
 	\draw (0,0) circle (0.3);
 	\draw (2,0) circle (0.3);
 	\draw (4,0) circle (0.3);
 	\draw (6,0) circle (0.3);
 	\draw (3,-1.2) circle (0.3);
    	\node[] at (0, 0) {1};
    	\node[] at (2, 0) {2};
    	\node[] at (4, 0) {$3$};
    	\node[] at (6, 0) {$5$};
    	\node[] at (3,-1.2) {$4$};
	 \draw[->] (0.35, .1) -- (1.65, .1);
	 \draw[->] (2.35,.1) -- (3.65,.1);
	 \draw[<-] (4.35,.1) -- (5.65,.1);
   	 \node[] at (1, 0.3) {$a_{21}$};
	\node[] at (3,0.3) {$a_{32}$};
	\node[] at (5,.3) {$a_{35}$};
	\draw (4,1) circle (0.07);	
	 \draw[-] (4, .3 ) -- (4, .93);
	 \draw[->] (6.65, .65) -- (6.25, .25);	
   	 \node[] at (6.8,.8) {in};
	 \draw[->] (-.65, .65) -- (-.25, .25);	
   	 \node[] at (-.8,.8) {in};
   	 
	 \draw[->] (2.2,-.3) -- (2.8, -.9);	
   	 \node[] at (2.2, -.7) {$a_{42}$};
   	 \draw[->] (3.2,-.9) -- (3.8,-.3);
   	 \node[] at (3.8,-.7) {$a_{34}$};
\draw (-1.2,-2) rectangle (7.2, 2);
    	\node[] at (0, -1.5) {$\cm'$};
    	

\end{tikzpicture}
    \caption{Sink Example}
    \label{fig:sink}
\end{figure}

These two models are indistinguishable via the following map:
\[
\begin{pmatrix}
a_{21} \\ a_{24} \\ a_{32} \\ a_{41} \\ a_{35} 
\end{pmatrix} \mapsto 
\begin{pmatrix}
a_{32} \\ a_{34} \\ a_{21} \\ a_{42} \\ a_{35}
\end{pmatrix}
\]



\section{Source Models} \label{section:source}


We now extend the idea of indistinguishability of path models to a family of models called \textit{source models}.  A source model is defined as a combination of skeletal path models which all of have same input, as depicted in Figure \ref{fig:basicsource}. 

\begin{figure}
    \centering
    \begin{tikzpicture}[scale=1]
 	\draw (0,0) circle (0.3);
 	\draw (-2,0) circle (0.3);
 	\draw (2,0) circle (0.3);
 	\draw (-1.4,-1.4) circle (0.3);
 	\draw (1.4,-1.4) circle (0.3);
 	\node[] at (-.3, -1.5) {$\cdots$};
 	\node[] at (.3, -1.5) {$\cdots$};

    	\node[] at (0, 0) {n};
    	\node[] at (-2, 0) {1};
    	\node[] at (-1.4, -1.4) {$i$};
    	\node[] at (1.4, -1.4) {$j$};
    	\node[] at (2,0) {$k$};
	 \draw[<-] (-1.7, 0) -- (-1.3, 0);
	 \node[] at (-1,0) {$\cdots$};
	 \draw[<-] (-.7,0) -- (-.32,0);
	 \draw[<-] (1.7, 0) -- (1.3, 0);
	 \node[] at (1,0) {$\cdots$};
	 \draw[<-] (.7,0) -- (.32,0);
    \draw[<-] (-1.18,-1.18) -- (-.9,-.9);
    \node at (-.85,-.85) {$\cdot$};
    \node at (-.7,-.7) {$\cdot$};
    \node at (-.55,-.55) {$\cdot$};
    \draw[<-] (-.5,-.5) -- (-.22,-.22);
    \draw[<-] (1.18,-1.18) -- (.9,-.9);
    \node at (.85,-.85) {$\cdot$};
    \node at (.7,-.7) {$\cdot$};
    \node at (.55,-.55) {$\cdot$};
    \draw[<-] (.5,-.5) -- (.22,-.22);

    \node at (0,1) {in};
    \draw[<-] (0, .35 ) -- (0, .8);

    \draw (-2.8,0) circle (0.07);
    \draw[-] (-2.73,0) -- (-2.3,0);
    \draw (2.8,0) circle (0.07);
    \draw[-] (2.73,0) -- (2.3,0);
     \draw[-] (-1.9,-1.9) -- (-1.62,-1.62);
     \draw (-1.95,-1.95) circle (0.07);
     \draw[-] (1.9,-1.9) -- (1.62,-1.62);
     \draw (1.95,-1.95) circle (0.07);
    
    \node at (0,1) {in};
    \draw[<-] (0, .35 ) -- (0, .8);

    	

\end{tikzpicture}
    \caption{Basic source model.}
    \label{fig:basicsource}
\end{figure}
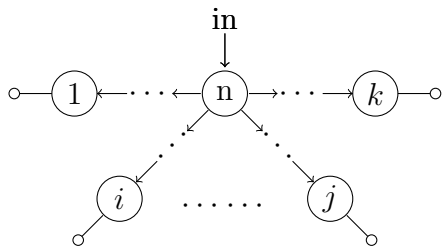

\begin{lemma} \label{lemma:swapinandout}
    The models $\cm=(G,S,\{i\},Leak)$ and $\cm'=(\tilde{G},\{i\},S,Leak)$, where the graph $\tilde{G}$ is the graph $G$ with the arrows reversed and $S$ is a subset of the vertex set $V$, have the same coefficient maps.
\end{lemma}

\begin{proof}
    Consider the input-output equation for the model $\cm$. The left hand side coefficients of the input-output equation are given by $\det(\partial I -A)$ and the right-hand side coefficients are given by $\det(\partial I - A)^{(j,i)}$ for each $j \in S$.  Now consider the model $\cm'$.  Since the graph $\tilde{G}$ is the graph $G$ but with the arrows reversed, this means the matrix $A$ is now $A^T$ for the model $\cm'$.  Since there are multiple outputs, we get one input-output equation for each output corresponding to $j \in S$.  However, the left-hand side for each of these input-output equations are given by $\det(\partial I - A^T) y_j$, thus the coefficients are identical on the left-hand side for each output.  The right-hand side coefficients are given by $\det(\partial I - A^T)^{(i,j)}$ for each $j \in S$.  This is precisely the transpose of $\det(\partial I - A)^{(j,i)}$, thus the coefficients are identical.
\end{proof}

\begin{cor} \label{cor:source}
Suppose two source models $\cm$ and $\cm'$ are identical outside of one path on each model which on their own are permutation indistinguishable, then $\cm$ and $\cm'$ are permutation indistinguishable.
\end{cor}

\begin{proof}
    This follows from Corollary \ref{cor:sink} and Lemma \ref{lemma:swapinandout}.
\end{proof}

\section{Discussion} \label{section:discussion}

In this work, we have considered a class of models containing a path from input to output and have found sufficient conditions based on the structure of the graph for the models to be indistinguishable.  To the best of our knowledge, these are the first sufficient conditions for indistinguishability of linear compartmental models based on graph structure alone.  We first showed that, given a path from compartment $1$ to $n$, with input in $1$ and output in $n$, the class of models with a leak from compartment $i$, where $i=1,..,n-1$, are all indistinguishable (see Theorem \ref{thm:leak}), as well as the model with an edge from $n$ to $n-1$ instead of a leak (see Theorem \ref{thm:terminalcycle}).  We then generalized this result to apply to moving any ``detour'' down the path (see Theorem \ref{thm:detour}).  Finally, we showed how to generalize further with multiple source-sink paths (see Corollary \ref{cor:sink} and Corollary \ref{cor:source}).

An obvious next step would be to consider models that do not contain this ``skeletal path'' structure.  For example, models that consist of cycles.  Another direction for future work is to consider indistinguishability not arising in the permutation indistinguishability framework that we have examined, but as in Example \ref{ex:nonpermute}.  

The original inspiration for this study of indistinguishability was the recent work of the authors of this text and collaborators on a novel graphical characterization of the coefficients of linear compartmental model input-output equations \cite{aim}. 
 From this graphical understanding of the input-output equations, as discussed in Section \ref{subsection:gfgeorules}, we had an immediate understanding of the  necessary conditions for indistinguishability.  Several of the proofs in this work started as graph theoretic arguments relating the coefficients of two indistinguishable models, and we suspect that this result, though not necessarily heavily utilized in this work, will be a strong tool in future works related to indistinguishability.


\bibliographystyle{plain}
\bibliography{AIM}

\end{document}